\documentclass[11pt]{article}

\usepackage{amsmath,amssymb,amsfonts,textcomp,amsthm,xifthen,psfrag,graphicx,color,MnSymbol}
\usepackage{subcaption,enumitem}
\usepackage[T1]{fontenc}
\usepackage{lscape}

\usepackage{hyperref}

\date{}

\newtheorem{theorem}{Theorem}
\newtheorem{lemma}[theorem]{Lemma}

\newtheorem{prop}[theorem]{Proposition}
\newtheorem{remark}[theorem]{Remark}

\theoremstyle{definition} 

\def\ignore#1{}

\def\be{\begin{equation}}
\def\ee{\end{equation}}
\def\ba{\begin{align}}
\def\ea{\end{align}}
\def\bas{\begin{align*}}
\def\eas{\end{align*}}

\newcommand{\cPF}{C_\mathrm{PF}}
\newcommand{\cPFh}{C_\mathrm{PF,d}}
\newcommand{\cis}{c_\mathrm{is}}
\newcommand{\Ipol}{\mathcal{I}}
\DeclareMathOperator{\opF}{\mathcal{F}}
\DeclareMathOperator{\opFpthree}{\mathcal{F}^{4}_b}
\DeclareMathOperator{\opFpthreea}{\mathcal{\widetilde F}^{4}_b}

\newcommand{\cF}{C_\mathrm{F}}

\newcommand{\Xpthreeloc}{\mathbb{X}^{4}_b}
\newcommand{\Xpthree}{\mathbb{X}^{4,c}_b}
\newcommand{\Xptwo}{\mathbb{X}^{3,\cN}}
\newcommand{\XdDiv}{\mathbb{X}^{\mathrm{dDiv}}}
\newcommand{\XdDivnn}{\mathbb{X}^\mathrm{dDiv}_\mathit{nn}}
\newcommand{\XdDivnnc}{\mathbb{X}^{\mathrm{dDiv},\mathit{nnc}}}
\newcommand{\HCT}{\mathit{HCT}}
\newcommand{\RT}{\mathit{RT}}

\newcommand{\R}{\mathbb{R}}
\newcommand{\Rt}{\mathbb{R}^2}

\newcommand{\N}{\mathbb{N}}
\renewcommand{\SS}{\mathbb{S}}

\newcommand{\cB}{\mathcal{B}}
\newcommand{\cE}{\mathcal{E}}
\newcommand{\cC}{\mathcal{C}}
\newcommand{\cCinv}{{\mathcal{C}^{-1}}}

\newcommand{\cU}{\mathcal{U}}
\newcommand{\cV}{\mathcal{V}}
\newcommand{\cN}{\mathcal{N}}

\newcommand{\cS}{\mathcal{S}}
\newcommand{\sH}{\widetilde H}

\newcommand{\bx}{{\boldsymbol{x}}}
\newcommand{\bn}{{\boldsymbol{n}}}
\newcommand{\bt}{{\boldsymbol{t}}}
\newcommand{\bu}{{\boldsymbol{u}}}
\newcommand{\bv}{{\boldsymbol{v}}}
\newcommand{\bw}{{\boldsymbol{w}}}
\newcommand{\bQ}{{\boldsymbol{Q}}}
\newcommand{\bM}{{\boldsymbol{M}}}
\newcommand{\bI}{{\boldsymbol{I}}}
\newcommand{\bpsi}{{\boldsymbol{\psi}}}
\newcommand{\bphi}{{\boldsymbol{\phi}}}
\newcommand{\boldeta}{{\boldsymbol{\eta}}}

\newcommand{\dbM}{{\boldsymbol{\delta\!M}}}

\newcommand{\du}{{\delta\!u}}
\newcommand{\dbu}{\boldsymbol{\delta\!u}}
\newcommand{\deta}{{\delta\!\eta}}
\newcommand{\dphi}{{\delta\!\phi}}
\newcommand{\dpsi}{{\delta\!\psi}}
\newcommand{\dbpsi}{{\delta\!\bpsi}}
\newcommand{\dbeta}{\boldsymbol{\delta\!\eta}}

\newcommand{\dv}{{\delta\!v}}
\newcommand{\dbv}{\boldsymbol{\delta\!v}}
\newcommand{\dw}{{\delta\!w}}

\newcommand{\<}{\langle{}}
\renewcommand{\>}{\rangle}
\newcommand{\ip}[2]{\llangle#1\hspace*{.5mm},#2\rrangle}
\newcommand{\dual}[2]{\<#1\hspace*{.5mm},#2\>}
\newcommand{\vdual}[2]{(#1\hspace*{.5mm},#2)}

\DeclareMathOperator{\sym}{sym}

\DeclareMathOperator{\dom}{dom}

\DeclareMathOperator{\trA}{\gamma_{A,\Gamma}}
\DeclareMathOperator{\trAS}{\gamma_{A,\cS}}
\DeclareMathOperator{\trASinv}{\gamma_{A,\cS}^{-1}}
\DeclareMathOperator{\trAsS}{\gamma_{A^*\!,\cS}}
\DeclareMathOperator{\trAsSinv}{\gamma_{A^*,\cS}^{-1}}

\newcommand{\grad}{\nabla}

\newcommand{\strain}{\varepsilon}

\renewcommand{\div}{\operatorname{div}}
\newcommand{\Div}{\operatorname{\mathbf{div}}}
\newcommand{\dDiv}{\operatorname{d\mathbf{div}}}

\newcommand{\Hnn}{H_{\mathit{nn}}}

\newcommand{\Idd}{\Pi^\mathrm{dDiv}}
\newcommand{\Iddnnc}{\Pi^{\mathrm{dDiv},\mathit{nnc}}}

\newcommand{\mesh}{\mathcal{T}}
\newcommand{\el}{T}

\newcommand{\trtwo}{{\gamma_{2}}}
\newcommand{\trtwoloc}{{\gamma_{2,\partial\el}}}
\newcommand{\trtone}{{\gamma_{2,1}}}

\newcommand{\trdDiv}{{\gamma_{\mathrm{dDiv}}}}

\newcommand{\trdDivM}{{\gamma_{\mathrm{dDiv},J0}}}
\newcommand{\HtrdDivM}{H^{-3/2,-1/2}_{J0}}
\newcommand{\trdnn}{{\gamma_{\mathrm{dDiv},\mathit{nn}}}}

\title{Generalized mixed and primal hybrid methods
with applications to plate bending
\thanks{Supported by ANID-Chile through FONDECYT project 1230013}
\author{
Norbert Heuer\thanks{
Facultad de Matem\'aticas, Pontificia Universidad Cat\'olica de Chile,
Avenida Vicu\~na Mackenna 4860, Santiago, Chile,
email: {\tt nheuer@uc.cl}}}}

\begin{document}
\maketitle
\begin{abstract}
We present an extended framework for hybrid finite element approximations
of self-adjoint, positive definite operators.
It covers the cases of primal, mixed, and ultraweak formulations, both at the
continuous and discrete levels, and gives rise to conforming discretizations.
Our framework allows for flexible continuity restrictions across elements,
and includes the extreme cases of conforming and discontinuous hybrid methods.
We illustrate an application of the framework to the Kirchhoff--Love
plate pending model and present three primal hybrid and two
mixed hybrid methods, four of them with numerical examples. In particular,
we present conforming frameworks for (in classical meaning) non-conforming elements of
Morley, Zienkiewicz triangular, and Hellan--Herrmann--Johnson types.


\end{abstract}

\section{Introduction}

We present a framework for generalized mixed and primal hybrid methods with flexible
continuity requirements of dual and primal variables.
It renders conforming otherwise non-conforming discretizations.
Our analysis is an extension of the classical Babu\v{s}ka--Brezzi framework
to a general hybrid approach where jump and trace (Lagrange multiplier)
terms are construed in the spirit of the discontinuous Petrov--Galerkin (DPG) method
with optimal test functions \cite{DemkowiczG_10_CDP,DemkowiczG_11_CDP}.
For special cases, the seminal papers on hybrid methods by
Brezzi, Marini, Raviart, and Thomas
\cite{Brezzi_75_MEF,BrezziM_75_NSP,Thomas_75_MEF,Thomas_76_MEF,RaviartT_77_PHF}
insinuate such embodiment of non-conformity. But only through the DPG setting
of trace and jump operations and their spaces, developed in recent years,
do we have a general abstract framework that settles well-posedness
at the continuous and discrete levels. This framework provides a conforming setting
for discrete schemes that are non-conforming in standard spaces.
We expect a conforming framework to be beneficial for a posteriori error
analysis, e.g., for normal-normal continuous stress approximations
in linear elasticity as proposed in \cite{PechsteinS_11_TDN,PechsteinS_18_ATM,CarstensenH_NNC}.
We prove well-posedness of the continuous and discrete
formulations, and quasi-optimality of the schemes.
In this paper, we restrict ourselves to general positive-definite
self-adjoint differential operators without lower-order terms, and
consider homogeneous Dirichlet boundary conditions.
We illustrate an application of the framework to the  Kirchhoff--Love plate bending model
which, in the simplest case, reduces to the biharmonic problem.
Rather than efficiency, the proposed discrete schemes are meant to
underline the generality and practicality of the framework for a non-trivial application.
In particular, our various trace interpretations show that individual components of
effective shear forces become accessible for stable approximations
by more canonical --generalized hybrid-- formulations.
This is relevant for engineering applications.

Different strategies are known to reduce the required regularity of discrete spaces
for mixed finite element methods both of the standard (dual) and primal types.
Such a regularity reduction is particularly attractive for spaces whose conformity
not only requires continuity of different orders but especially for those aiming at
pointwise symmetry. Examples are stress tensors of linear elasticity and bending moments
in plate bending, but also (though without symmetry) deflections in plate bending
of Kirchhoff--Love type. Strategies include discontinuous Galerkin (DG), non-conforming,
and hybrid methods, see, e.g.,
\cite{Ciarlet,BrezziF_91_MHF,ArnoldBCM_02_UAD,Riviere_08_DGM,DiPietroE_12_MAD}.
This has been and continues to be a very active field of research, and an up-to-date
literature discussion goes beyond the scope of this paper.
DG and non-conforming methods are usually defined and analyzed at a purely discrete level
in finite-dimensional spaces.
Hybrid methods can also be interpreted as being of the discontinuous Galerkin type,
but their setting is usually closer to the underlying variational formulation of the problem.
Primal hybrid methods ignore the conformity of approximations of primal variables across
element interfaces, and compensate this by the introduction of Lagrange multipliers.
Mixed hybrid settings do this for the dual variable.
The idea of using non-conforming hybrid discretizations goes back to the
solid mechanics community, see Fraeijs de Veubeke \cite{FraeijsdV_65_DEM} and
Pian \& Tong, Pian \cite{PianT_69_BFE,Pian_72_FEF}.
Thomas and Raviart \& Thomas provided analyses of the dual and primal hybrid schemes
\cite{Thomas_75_MEF,Thomas_76_MEF,RaviartT_77_PHF}.
See also the early book by Brezzi and Fortin \cite{BrezziF_91_MHF},
predecessor of \cite{BoffiBF_13_MFE}.
Brezzi provided an abstract framework for the appearing saddle point problems and discretizations
in his seminal paper \cite{Brezzi_74_OEU} and,
also together with Marini \cite{Brezzi_75_MEF,BrezziM_75_NSP},
analyzed such schemes for the biharmonic problem.
Of course, there is related analysis for saddle point discretizations by
Babu\v{s}ka \cite{Babuska_73_FEM}, and Ladyzhenskaya \cite{Ladyzhenskaya_69_MTV}
is usually cited for the continuous analysis. Brezzi specifically motivated
his formulation and conditions for the analysis of hybrid finite element schemes.

There is no clear separation between different types of methods with approximations
of reduced regularity, cf. the discussions in \cite{ArnoldB_85_MNF,BlumR_90_MFE},
see also \cite[Remark~10.3.4]{BoffiBF_13_MFE}.
In this paper, we consider a finite element discretization that is non-conforming
in standard spaces to be of hybrid type if there is a variational formulation
that renders the discretization conforming.
We propose and analyze a framework both for primal and mixed hybrid methods
with general continuity (conformity) requirements so as to have a functional setting
that covers the extreme cases of primal/mixed Galerkin methods and primal/mixed hybrid methods,
and includes intermediate cases. We also consider the case of ultraweak formulations.
As mentioned before, our trace and jump analysis is in the spirit of the
discontinuous Petrov--Galerkin (DPG) method with optimal test functions, proposed
by Demkowicz and Gopalakrishnan as a scheme that aims at ``automatic'' discrete inf-sup
stability \cite{DemkowiczG_10_CDP,DemkowiczG_11_CDP}. The DPG method is typically
based on ultraweak formulations. Their use has been proposed by Despr\'es and Cessenat
\cite{Despres_94_SUF,CessenatD_98_AUW}. Switching to an ultraweak formulation
means that all appearing derivatives are thrown onto the test side via integrations
by parts, in this way generating trace terms.
Bottasso \emph{et. el.} \cite{BottassoMS_02_DPG} did this in a DG setting
and suggested to replace trace terms with independent variables,
as proposed in the unified DG setting of \cite{ArnoldBCM_02_UAD}.
As the DPG setting is essentially a functional analytic framework, the use of ultraweak
formulations requires a definition of trace variables in continuous spaces.
This has led to a specific analysis of such formulations
\cite{CarstensenDG_16_BSF,FuehrerHS_20_UFR}, see also
\cite[Appendix A]{DemkowiczGNS_17_SDM} for a first abstract setting.
In this paper, we establish that the treatment of traces in the ultraweak DPG setting gives rise to
a generalized framework for hybrid methods that usually work with quotient spaces
and ``harmonic'' extensions. We again refer to Brezzi \cite{Brezzi_74_OEU} for abstract
saddle point formulations, there motivated by the analysis of hybrid schemes,
and to the more extensive treatise \cite{BrezziF_91_MHF} for specific problems.
We systematically work with trace spaces and norms as in the DPG spirit,
but of a more flexible scope, and provide abstract formulations for hybrid settings.
The well-posedness of our variational formulations and their discretizations
follows from the Brezzi theory by verifying the now standard criteria.

An overview of the remainder is as follows.
In Section~\ref{sec_abstract} we present the abstract framework of our variational formulations
and their discretizations.
We do this for general self-adjoint positive definite operators but,
for ease of presentation, restrict ourselves to homogeneous Dirichlet conditions and discard
lower-order terms. The primal variable is assumed to be scalar, though an extension of
our framework to vector cases is straightforward.
In Subsections~\ref{sec_abstract_gp} and~\ref{sec_abstract_gm}
we prove the well-posedness of generalized primal hybrid and mixed hybrid formulations,
respectively, and the quasi-optimality of their discretizations.
For its relevance for the DPG method, we also study a general ultraweak formulation
in Subsection~\ref{sec_abstract_gu}. A direct discretization would be a Petrov--Galerkin
scheme, as used for the DPG method with optimal test functions.
We briefly illustrate the path of a minimum residual discretization which is
analogous to a) a fully discrete DPG scheme as proposed in \cite{GopalakrishnanQ_14_APD}
and b) a variational stabilization proposed in \cite{CohenDW_12_AVS}. Proofs of all the
abstract results are given at the end, in Section~\ref{sec_abstract_proofs}.
In Section~\ref{sec_KL} we illustrate an application of the abstract results
to the Kirchhoff--Love plate bending problem, and present five different
formulations with discretizations. We include the primal and mixed hybrid versions
(in \S\ref{sec_KL_p1} and \S\ref{sec_KL_d1}, respectively)
which are new formulations, to the best of our knowledge,
but are equivalent to the ones proposed in \cite{Brezzi_74_OEU,BrezziM_75_NSP},
though with different discretizations.
In the remaining Sections~\ref{sec_KL_p2},~\ref{sec_KL_p3}, and~\ref{sec_KL_d2} we depart
from the standard hybrid settings and consider formulations that do impose
continuity restrictions between elements of principal variables. To our knowledge,
these formulations and discretizations are new. In particular, we do not
assume convexity of the domain and appearing traces of bending moments include corner
forces in a well-posed manner.
A primal hybrid scheme with continuity at vertices is the subject of \S\ref{sec_KL_p2}.
It can be interpreted as a scheme with Morley-type element and Lagrangian multipliers,
cf.~\cite{Morley_68_TEE}. A primal hybrid scheme with continuous approximation is studied
in \S\ref{sec_KL_p3}. It can be seen as a conforming extension of
$C^0$-interior penalty (C0IP) schemes, see, e.g., \cite{EngelGHLMT_02_CDF,BrennerS_05_CIP},
or the Zienkiewicz triangular element, cf.~\cite{BazeleyCIZ_65_TEP,LascauxL_75_SNF},
see also the hybrid high-order (HHO) method in \cite{DongE_24_CHH} with continuous
approximation and the C0-hybrid formulation in \cite{ChenH_NDD}, the latter
also presenting a mixed hybrid discretization.
In \S\ref{sec_KL_d2} we study a mixed hybrid formulation and discretization that imposes
normal-normal continuity of bending moments, and thus is a conforming framework
for elements of the Hellan--Herrmann--Johnson (HHJ) type,
cf.~\cite{Hellan_67_AEP,Herrmann_67_FEB,Johnson_73_CMF}.
Our abstract framework covers the cases of classical
primal and mixed formulations, not explictly studied here.
In the case of the Kirchhoff--Love model, composite
Hsieh--Clough--Tocher (HCT) elements can be used for the conforming approximation of
primal variables, see~\cite{CloughT_65_FES,Ciarlet_78_IEE,DouglasDPS_79_FCF},
whereas conforming elements for bending moments and the resulting mixed method
have been recently presented in \cite{FuehrerH_MKL}. We do use traces of HCT elements
to approximate traces of $H^2$-variables, and also use a reduced element from
\cite{FuehrerH_MKL} to approximate bending moments and their traces.
All our discretizations are of low order and aim at a low number of degrees of freedom,
though we do not claim optimality. For all but the primal hybrid method
(which is very close to the nodal-continuous primal hybrid method) we present numerical
results that illustrate the convergence properties of the schemes for a smooth
model solution, see Section~\ref{sec_KL_num}.
In this paper, we generally prove quasi-optimal error estimates in energy norms, and leave
the proof of specific convergence orders open. We also do not elaborate on 
superconvergence results which can be observed in some cases.
We restrict our analysis to homogeneous Dirichlet boundary conditions.
Though, we stress the fact that the analysis can be extended to include any
combination of different types of boundary conditions that make physical sense.
This is due to the fact that all the variables and spaces stem from well-posed
variational formulations: no non-physical terms or Lagrangian multipliers are used.
For an illustration of the inclusion of boundary conditions we refer to
\cite{CarstensenH_NNC} which deals with the particular case of plane elasticity.
At an abstract level, our framework includes the analysis provided there.

\section{Abstract framework} \label{sec_abstract}

For a bounded Lipschitz domain $\Omega\subset\R^d$ ($d\in\N$)
and $U\in\{\R,\R^d,\R^{d\times d}\}$ or $U=\SS:=\{w\in \R^{d\times d};\; w=w^\top\}$
we consider a linear unbounded $U$-valued differential operator
$A:\;\dom(A)\subset L_2(\Omega)\to L_2(\Omega;U)$ with constant coefficients and
unbounded formal adjoint $A^*:\;\dom(A^*)\subset L_2(\Omega;U)\to L_2(\Omega)$,
and the Hilbert spaces with (squared) norms
\begin{align*}
   H(A)&:=\{v\in L_2(\Omega);\; Av\in L_2(\Omega;U)\},
   &&\|v\|_A^2 := \|v\|^2 + \|Av\|^2,\\
   H(A^*)&:=\{w\in L_2(\Omega;U);\; A^*w\in L_2(\Omega)\},
   &&\|w\|_{A^*}^2 := \|w\|^2 + \|A^*w\|^2.
\end{align*}
Here, $\|\cdot\|$ is the generic $L_2(\Omega)$-norm and, below, $\vdual{\cdot}{\cdot}$
denotes the generic $L_2(\Omega)$ duality.
There is a corresponding trace operator
\begin{align*}
   \trA:\; \left\{\begin{array}{cll}
           H(A) &\rightarrow& H(A^*)^*,\\
           v &\mapsto&
               \dual{\trA(v)}{w}_\Gamma :=
               \vdual{Av}{w} - \vdual{v}{A^*w}
           \end{array}\right.
\end{align*}
and space with vanishing traces
\[
   H_0(A):=\{v\in H(A);\; \trA(v)=0\}.
\]
Given a symmetric, positive definite tensor-field
$\cC\in L_\infty(\Omega;U\times U)$ and $f\in L_2(\Omega)$,
our model problem with homogeneous Dirichlet boundary condition reads
\begin{align} \label{prob}
   u\in H_0(A):\quad A^*\cC Au = f.
\end{align}
There are three canonical variational formulations, the Euler--Lagrange equation
\begin{align} \label{EL}
   u\in H_0(A):\quad \vdual{\cC Au}{A\du} = \vdual{f}{\du}\quad\forall\du\in H_0(A)
\end{align}
and, with independent variable $w:=\cC Au$,
the primal mixed formulation
\begin{subequations} \label{p}
\begin{alignat}{4}
   w\in L_2(\Omega;U),\ u\in H_0(A):\qquad
   &\vdual{\cCinv w}{\dw} - \vdual{A u}{\dw}
   &&= 0 &&\forall\dw\in L_2(\Omega;U), \\
   &-\vdual{w}{A\du} &&= -\vdual{f}{\du}\quad &&\forall\du\in H_0(A)
\end{alignat}
\end{subequations}
and the (dual) mixed formulation
\begin{subequations} \label{d}
\begin{alignat}{4}
   w\in H(A^*),\ u\in L_2(\Omega):\qquad
   &\vdual{\cCinv w}{\dw} - \vdual{u}{A^*\dw}
   &&= 0 &&\forall\dw\in H(A^*), \\
   & -\vdual{A^*w}{\du} &&= -\vdual{f}{\du}\quad &&\forall\du\in L_2(\Omega).
\end{alignat}
\end{subequations}
Under the standard assumptions
\begin{subequations} \label{ass}
\begin{alignat}{4}
   &\exists \cPF>0: &&\|v\|\le \cPF \|Av\|  &&\quad\forall v\in H_0(A)
                 &&\quad\text{(Poincar\'e--Friedrichs)}, \label{PF}\\
   &\exists \cis>0: \sup_{w\in H(A^*),\; \|w\|_{A^*}=1}
                 && \vdual{A^*w}{v} \ge \cis\|v\|   &&\quad\forall v\in L_2(\Omega)
                 &&\quad\text{(inf-sup)}, \label{infsup}
\end{alignat}
\end{subequations}
formulations \eqref{EL}, \eqref{p}, \eqref{d} are well posed and equivalent.
We note that the chosen boundary condition renders \eqref{PF}, \eqref{infsup} equivalent.

Our aim is to provide well-posed formulations that require less regularity than
$u\in H(A)$ and $w\in H(A^*)$.
To this end we consider a regular mesh $\mesh=\{\el\}$ of polyhedrals $\el$ covering
$\Omega$ and introduce the product spaces with (squared) norms
\begin{align*}
   H(A,\mesh)&:=\{v\in L_2(\Omega);\; Av|_\el\in L_2(\el;U)\ \forall\el\in\mesh\},
   &&\|v\|_{A,\mesh}^2 := \|v\|^2 + \|Av\|_\mesh^2,\\
   H(A^*,\mesh)&:=\{w\in L_2(\Omega;U);\; A^*w|_\el\in L_2(\el)\ \forall\el\in\mesh\},
   &&\|w\|_{A^*,\mesh}^2 := \|w\|^2 + \|A^*w\|_\mesh^2.
\end{align*}
Here, $\|\cdot\|_\mesh^2:=\vdual{\cdot}{\cdot}_\mesh$ where the latter indicates
the $L_2(\mesh)$-duality.
In the following, $A_\mesh$ and $A^*_\mesh$ denote the
corresponding $\mesh$-piecewise differential operators, e.g.,
$\|Av\|_\mesh=\|A_\mesh v\|$ for any $v\in H(A,\mesh)$.
Here and in the following, we identify elements of product spaces with corresponding
piecewise defined function.
For instance, $(v_\el)_{\el\in\mesh}\in \Pi_{\el\in\mesh} H(A,\el)$ is identified with
$v\in L_2(\Omega)$ defined by $v|_\el:=v_\el$, $\el\in\mesh$.

We consider closed spaces $\sH(A,\mesh)\subset H(A,\mesh)$ and
$\sH(A^*,\mesh)\subset H(A^*,\mesh)$ that are intermediate:
\begin{align*}
   H_0(A)\subset \sH(A,\mesh)\subset H(A,\mesh),\quad
   H(A^*)\subset \sH(A^*,\mesh)\subset H(A^*,\mesh).
\end{align*}
These spaces induce two trace operators with support on the skeleton
$\cS=\cup\{\partial\el;\el\in\mesh\}$,
\begin{subequations}
\begin{align}
   \label{trAS}
   \trAS:\; &\left\{\begin{array}{cll}
           H(A) &\rightarrow& \sH(A^*,\mesh)^*,\\
           v &\mapsto&
               \dual{\trAS(v)}{w}_\cS :=
               \vdual{Av}{w} - \vdual{v}{A^*w}_\mesh
           \end{array}\right.,\\
   \label{trAsS}
   \trAsS:\; &\left\{\begin{array}{cll}
           H(A^*) &\rightarrow& \sH(A,\mesh)^*,\\
           w &\mapsto&
               \dual{\trAsS(w)}{v}_\cS :=
               \vdual{A^*w}{v} - \vdual{w}{Av}_\mesh
           \end{array}\right.,
\end{align}
\end{subequations}
and corresponding trace spaces
\[
   H(A,\cS) := \trAS(H(A)),\quad H_0(A,\cS) := \trAS(H_0(A)),\quad H(A^*,\cS) := \trAsS(H(A^*))
\]
with norms
\begin{align*}
   &\|\phi\|_{A,\cS} := \inf\{\|v\|_A;\; v\in H(A),\ \trAS(v)=\phi\},\\
   &\|\phi\|_{(A^*,\sim,\mesh)^*} := \sup_{w\in \sH(A^*,\mesh),\; \|w\|_{A^*,\mesh}=1}
                               \dual{\phi}{w}_\cS\qquad (\phi\in H(A,\cS)),\\
   &\|\psi\|_{A^*,\cS} := \inf\{\|w\|_{A^*};\; w\in H(A^*),\ \trAsS(w)=\psi\},\\
   &\|\psi\|_{(A,\sim,\mesh)^*} := \sup_{v\in \sH(A,\mesh),\; \|v\|_{A,\mesh}=1}
                               \dual{\psi}{v}_\cS\qquad (\psi\in H(A^*,\cS)).
\end{align*}
Here, the dualities are defined as
\begin{align*}
   &\dual{\phi}{w}_\cS := \dual{\trAS(v)}{w}_\cS\quad (w\in \sH(A^*,\mesh),\; v\in\trASinv(\phi))\\
   \text{and}\quad
   &\dual{\psi}{v}_\cS := \dual{\trAsS(w)}{v}_\cS\quad (v\in \sH(A,\mesh),\; w\in\trAsSinv(\psi)).
\end{align*}
Of course, in the case that
$\sH(A^*,\mesh)=H(A^*)$, $\trAS=\trA$.

\subsection{Generalized primal hybrid formulation} \label{sec_abstract_gp}

To reformulate \eqref{prob}, we define $w:=\cC Au$. 
Testing $A^*w=f$  with $\du\in\sH(A,\mesh)$ and applying trace operator $\trAsS$,
we find that
\begin{align*}
   \vdual{f}{\du} = \vdual{w}{A\du}_\mesh + \dual{\trAsS(w)}{\du}_\cS.
\end{align*}
We introduce the independent variable $\psi:=\trAsS(w)$ and relax the regularity of $u$
in the weak form of $\cCinv w=Au$.
This leads to a generalized primal-mixed formulation of \eqref{prob}:
\emph{Find $w\in L_2(\Omega;U)$, $u\in\sH(A,\mesh)$, and $\psi\in H(A^*,\cS)$ such that}
\begin{subequations} \label{gp_full}
\begin{alignat}{3}
   &\vdual{\cCinv w}{\dw} - \vdual{A u}{\dw}_\mesh - \dual{\dpsi}{u}_\cS
   &&= 0
   &&\forall\dw\in L_2(\Omega;U),\ \forall\dpsi\in H(A^*,\cS), \label{gpa_full}\\
   &-\vdual{w}{A\du}_\mesh - \dual{\psi}{\du}_\cS
   &&= -\vdual{f}{\du}\quad
   &&\forall\du\in\sH(A,\mesh). \label{gpb_full}
\end{alignat}
\end{subequations}
Elimination of $w$ yields the \emph{generalized primal hybrid formulation}:
\emph{Find $u\in\sH(A,\mesh)$ and $\psi\in H(A^*,\cS)$ such that}
\begin{subequations} \label{gp}
\begin{alignat}{3}
   &\vdual{\cC A u}{A\du}_\mesh + \dual{\psi}{\du}_\cS
   &&= \vdual{f}{\du}\quad
   &&\forall\du\in\sH(A,\mesh), \label{gpa}\\
   & \dual{\dpsi}{u}_\cS &&= 0
   &&\forall\dpsi\in H(A^*,\cS). \label{gpb}
\end{alignat}
\end{subequations}

\begin{theorem} \label{thm_gp}
Let $f\in L_2(\Omega)$ be given and assume that \eqref{PF} holds.
Problem \eqref{gp} is well posed. Its solution $(u,\psi)$ satisfies
\[
   \|u\|_{A,\mesh} + \|\psi\|_{A^*,\cS} \le C \|f\|
\]
with a constant $C$ that is independent of $f$ and $\mesh$.
Furthermore, $u\in H_0(A)$, $\psi=\trAsS(\cC Au)$, and $u$ solves \eqref{prob}.
\end{theorem}

A proof of this theorem is given in \S\ref{pf_thm_gp}.

For a discretization of \eqref{gp} we select finite-dimensional subspaces
$\sH_h(A,\mesh)\subset \sH(A,\mesh)$, $H_h(A^*,\cS)\subset H(A^*,\cS)$,
assume the existence of a Fortin operator $\opF:\;\sH(A,\mesh)\to \sH_h(A,\mesh)$,
\begin{subequations} \label{Fgp}
\begin{alignat}{2}
   &\dual{\dpsi}{v-\opF v}_\cS=0
   &&\forall \dpsi\in H_h(A^*,\cS),\ \forall v\in \sH(A,\mesh),\\
   \exists \cF>0:\quad &\|\opF v\|_{A,\mesh}\le \cF \|v\|_{A,\mesh}\quad
   &&\forall v\in \sH(A,\mesh),
\end{alignat}
\end{subequations}
and the validity of the discrete Poincar\'e--Friedrichs inequality
\begin{align} \label{PFh}
   &\exists \cPFh>0: \quad\|v\|\le \cPFh \|Av\|_\mesh\quad  \forall v\in \sH_h(A,\mesh)
   \quad\text{with}\quad
   \dual{\dpsi}{v}_\cS=0\quad \forall \dpsi\in H_h(A^*,\cS)
\end{align}
to conclude the well-posedness and quasi-optimal convergence of the
\emph{generalized primal hybrid method}:
\emph{Find $u_h\in\sH_h(A,\mesh)$ and $\psi_h\in H_h(A^*,\cS)$ such that}
\begin{subequations} \label{gph}
\begin{alignat}{3}
   &\vdual{\cC A u_h}{A\du}_\mesh + \dual{\psi_h}{\du}_\cS
   &&= \vdual{f}{\du}\quad
   &&\forall\du\in\sH_h(A,\mesh), \label{gpha}\\
   & \dual{\dpsi}{u_h}_\cS &&= 0
   &&\forall\dpsi\in H_h(A^*,\cS). \label{gphb}
\end{alignat}
\end{subequations}

\begin{theorem} \label{thm_gph}
Under the conditions of Theorem~\ref{thm_gp},
and assuming that \eqref{Fgp} and \eqref{PFh} hold with constants $\cF$, $\cPFh$
independent of $\mesh$ and the discrete subspace,
scheme \eqref{gph} has a unique solution $(u_h,\psi_h)$. It satisfies
\[
   \|u-u_h\|_{A,\mesh} + \|\psi-\psi_h\|_{A^*,\cS}
   \le C \bigl(\|u-v\|_{A,\mesh} + \|\psi-\eta\|_{A^*,\cS}\bigr)
   \quad\forall v\in \sH_h(A,\mesh),\ \forall \eta\in H_h(A^*,\cS)
\]
with a constant $C$ that is independent of $f$, $\mesh$ and the discrete subspaces.
\end{theorem}

A proof of this theorem is standard, cf.~, e.g.,\
\cite[Theorem~5.2.5,~Proposition~5.4.2]{BoffiBF_13_MFE}.

\subsection{Generalized mixed hybrid formulation} \label{sec_abstract_gm}

As before, we introduce $w:=\cC Au$. Testing this relation with $\dw\in\sH(A^*,\mesh)$
and applying trace operator $\trAS$, we find that
\begin{align*}
   \vdual{\cCinv w}{\dw} = \vdual{u}{A^*\dw}_\mesh + \dual{\trAS(u)}{\dw}_\cS.
\end{align*}
We introduce the independent trace variable $\phi:=\trAS(u)$ and
relax the regularity of $w$ in the weak form of relation $A^*w=f$.
This leads to the \emph{generalized mixed hybrid formulation} of \eqref{prob}:
\emph{Find $w\in\sH(A^*,\mesh)$, $u\in L_2(\Omega)$, and $\phi\in H_0(A,\cS)$ such that}
\begin{subequations} \label{gm}
\begin{alignat}{3}
   &\vdual{\cCinv w}{\dw} - \vdual{u}{A^*\dw}_\mesh - \dual{\phi}{\dw}_\cS
   &&= 0
   &&\forall\dw\in\sH(A^*,\mesh), \label{gma}\\
   &-\vdual{A^*w}{\du}_\mesh - \dual{\dphi}{w}_\cS
   &&= -\vdual{f}{\du}\quad
   &&\forall\du\in L_2(\Omega),\ \forall\dphi\in H_0(A,\cS).\label{gmb}
\end{alignat}
\end{subequations}

\begin{theorem} \label{thm_gm}
Let $f\in L_2(\Omega)$ be given and assume that relations \eqref{ass} hold.
Problem \eqref{gm} is well posed. Its solution $(w,u,\phi)$ satisfies
\[
   \|w\|_{A^*,\mesh} + \|u\| + \|\phi\|_{A,\cS} \le C \|f\|
\]
with a constant $C$ that is independent of $f$ and $\mesh$.
Furthermore, $u\in H_0(A)$, $w=\cC Au\in H(A^*)$, $\phi=\trAS(u)$, and $u$ solves \eqref{prob}.
\end{theorem}

A proof of this theorem in given in \S\ref{pf_thm_gm}.

For a discretization of \eqref{gm} we select finite-dimensional subspaces
$\sH_h(A^*,\mesh)\subset \sH(A^*,\mesh)$, $H_h(\mesh)\subset L_2(\Omega)$,
$H_h(A,\cS)\subset H_0(A,\cS)$, assume that there are operators
\begin{subequations} \label{Fgm}
\begin{align}
   &\opF_1:\:\;\sH(A^*,\mesh)\to \sH_h(A^*,\mesh)\cap H(A^*),\\
   &\opF_2:\;\sH(A^*,\mesh)\to \sH_h(A^*,\mesh)
\end{align}
that satisfy
\begin{alignat}{2}
   &\vdual{A^*(w-\opF_1 w)}{\du}_\mesh=0\quad
   &&\forall \du\in H_h(\mesh),\ \forall w\in \sH(A^*,\mesh), \label{Fgm1}\\
   &\dual{\dphi}{w-\opF_2 w}_\cS=0\quad
   &&\forall \dphi\in H_h(A,\cS),\ \forall w\in \sH(A^*,\mesh), \label{Fgm2}\\
   \exists C_1>0:\quad &\|\opF_1 w\|_{A^*,\mesh}\le C_1 \|w\|_{A^*,\mesh}
   &&\forall w\in \sH(A^*,\mesh), \label{Fgm3}\\
   \exists C_2>0:\quad &\|\opF_2 w\|_{A^*,\mesh}\le C_2 \|w\|_{A^*,\mesh}
   &&\forall w\in \sH(A^*,\mesh), \label{Fgm4}
\end{alignat}
\end{subequations}
and the validity of relation
\begin{align} \label{subh}
   A^*_\mesh \bigl(\sH(A^*,\mesh)\bigr)\subset H_h(\mesh)
\end{align}
to conclude the well-posedness and quasi-optimal convergence of the
\emph{generalized mixed hybrid method}:
\emph{Find $w_h\in\sH_h(A^*,\mesh)$, $u_h\in H_h(\mesh)$, and $\phi_h\in H_h(A,\cS)$ such that}
\begin{subequations} \label{gmh}
\begin{alignat}{3}
   &\vdual{\cCinv w_h}{\dw} - \vdual{u_h}{A^*\dw}_\mesh - \dual{\phi_h}{\dw}_\cS
   &&= 0
   &&\forall\dw\in\sH_h(A^*,\mesh), \label{gmha}\\
   &-\vdual{A^*w_h}{\du}_\mesh - \dual{\dphi}{w_h}_\cS
   &&= -\vdual{f}{\du}\quad
   &&\forall\du\in H_h(\mesh),\ \forall\dphi\in H_h(A,\cS).\label{gmhb}
\end{alignat}
\end{subequations}
Properties \eqref{Fgm} imply that a discrete inf-sup condition holds, as we establish now.
This result is a discrete analogue of \cite[Theorem~3.3]{CarstensenDG_16_BSF} which considers
the continuous inf-sup condition.

\begin{prop} \label{prop_split_h}
If operators $\opF_1$, $\opF_2$ with properties \eqref{Fgm} exist, then any
$\du\in H_h(\mesh)$ and $\dphi\in H_h(A,\cS)$ are bounded as
\begin{equation} \label{infsup_h}
   \|\du\| + \|\dphi\|_{A,\cS}
   \le  (C_2 + \frac {C_1}{\cis} (C_2+1))
        \sup_{w\in\sH_h(A^*,\mesh)\setminus\{0\}}
        \frac {\vdual{A^* w}{\du}_\mesh + \dual{\dphi}{w}_\cS}{\|w\|_{A^*,\mesh}}.
\end{equation}
\end{prop}

A proof of this proposition is given in \S\ref{pf_prop_split_h}.
Let us state the well-posedness and quasi-optimal convergence of the
generalized mixed hybrid method.

\begin{theorem} \label{thm_gmh}
Under the conditions of Theorem~\ref{thm_gm}, assuming \eqref{subh} and the existence of
operators $\opF_1$, $\opF_2$ that satisfy \eqref{Fgm} with constants $C_1$, $C_2$
independent of $\mesh$ and the discrete subspaces,
scheme \eqref{gmh} has a unique solution $(w_h,u_h,\phi_h)$. It satisfies
\[
   \|w-w_h\|_{A^*,\mesh} + \|u-u_h\| + \|\phi-\phi_h\|_{A,\cS}
   \le C \bigl(\|w-z\|_{A^*,\mesh} + \|u-v\| + \|\phi-\eta\|_{A,\cS}\bigr)
\]
for any $z\in \sH_h(A^*,\mesh)$, $v\in H_h(\mesh)$, and $\eta\in H_h(A,\cS)$
with a constant $C$ that is independent of $f$, $\mesh$ and the discrete subspaces.
\end{theorem}

The comment to the proof of Theorem~\ref{thm_gph} applies in this case as well,
noting that the discrete inf-sup property is satisfied due to Proposition~\ref{prop_split_h}
and that relation \eqref{subh} implies the uniform coercivity
\[
   \vdual{\cCinv w}{w}\gtrsim \|w\|_{A^*,\mesh}^2\quad
   \forall w\in \sH_h(A^*,\mesh)\ \text{with}\
   \vdual{A^* w}{\du}=0\ \forall \du\in H_h(\mesh).
\]

\subsection{General ultraweak formulation} \label{sec_abstract_gu}

For completeness we also consider a general ultraweak formulation of \eqref{prob}
and show its well-posedness.
We introduce $w:=\cC Au$ and combine trace relations \eqref{gpb_full} and \eqref{gma}
with the two independent trace variables $\phi:=\trAS(u)$, $\psi:=\trAsS(w)$.
This gives the \emph{general ultraweak formulation}:
\emph{Find $u\in L_2(\Omega)$, $w\in L_2(\Omega;U)$, $\phi\in H_0(A,\cS)$,
and $\psi\in H(A^*,\cS)$ such that}
\begin{subequations} \label{gu}
\begin{alignat}{3}
   &\vdual{\cCinv w}{\dw} - \vdual{u}{A^*\dw}_\mesh - \dual{\phi}{\dw}_\cS
   &&= 0
   &&\forall\dw\in\sH(A^*,\mesh), \label{gua}\\
   &-\vdual{w}{A\du}_\mesh - \dual{\psi}{\du}_\cS
   &&= -\vdual{f}{\du}\quad
   &&\forall\du\in\sH(A,\mesh). \label{gub}
\end{alignat}
\end{subequations}
Of course, $w$ can be eliminated but this requires to introduce a test space of higher regularity
than we are considering here.

\begin{theorem} \label{thm_gu}
Let $f\in L_2(\Omega)$ be given and assume that relations \eqref{ass} hold.
Problem \eqref{gu} is well posed. Its solution $(w,u,\phi)$ satisfies
\[
   \|u\| + \|w\| + \|\phi\|_{A,\cS} + \|\psi\|_{A^*,\cS} \le C \|f\|
\]
with a constant $C$ that is independent of $f$ and $\mesh$.
Furthermore, $u\in H_0(A)$, $w=\cC Au\in H(A^*)$, $\phi=\trAS(u)$, $\psi=\trAsS(w)$,
and $u$ solves \eqref{prob}.
\end{theorem}

A proof of this theorem is given in \S\ref{pf_thm_gu}.

System \eqref{gu} is unsymmetric and any direct discretization would be
a Petrov--Galerkin scheme, the prototype being the discontinuous Petrov--Galerkin (DPG)
method with optimal test functions \cite{DemkowiczG_14_ODM}.
We aim at symmetric formulations and extend \eqref{gu} to a symmetric mixed system
by introducing the residual as independent variable.
This is in fact the stabilization idea of Dahmen et al.,
see~\cite{CohenDW_12_AVS,DahmenHSW_12_APG}, and equivalent to the DPG approach,
see~\cite[(2.21)]{DemkowiczG_14_ODM}.

To formulate the extended mixed system we
introduce the linear functional $L((\du,\dw)):=-\vdual{f}{\du}$ and
abbreviate the bilinear form from \eqref{gu} by
\[
   b(\bu,\dbv):=
   \vdual{\cCinv w}{\dw} - \vdual{u}{A^*\dw}_\mesh - \dual{\phi}{\dw}_\cS
   -\vdual{w}{A\du}_\mesh - \dual{\psi}{\du}_\cS
\]
with
\begin{alignat*}{2}
   &\bu=(u,w,\phi,\psi)
   &&\in\cU(\mesh):=L_2(\Omega)\times L_2(\Omega;U)\times H_0(A,\cS)\times H(A^*,\cS)\\
   \text{and}\quad
   &\dbv=(\du,\dw)
   &&\in\cV(\mesh):=\sH(A,\mesh)\times \sH(A^*,\mesh).
\end{alignat*}
Let $\ip{\cdot}{\cdot}_{\cV(\mesh)}$ denote the inner product of $\cV(\mesh)$
with norm $\|\cdot\|_{\cV(\mesh)}$.
For $\bu\in\cU(\mesh)$, we introduce the Riesz representation $\bv\in\cV(\mesh)$ of the
(negative) residual of \eqref{gu},
\begin{align*}
   \ip{\bv}{\dbv}_{\cV(\mesh)} = L(\dbv)-b(\bu,\dbv)\ \forall\dbv\in\cV(\mesh)
   \quad\text{and}\quad
   \|\bv\|_{\cV(\mesh)} = \|b(\bu,\cdot)-L\|_{\cV(\mesh)^*}.
\end{align*}
This leads to the (trivial) mixed representation of \eqref{gu}:
\emph{Find $\bv\in\cV(\mesh)$ and $\bu\in\cU(\mesh)$ such that}
\begin{subequations} \label{gum}
\begin{alignat}{3}
   &\ip{\bv}{\dbv}_{\cV(\mesh)} + b(\bu,\dbv) &&= L(\dbv)\quad
   &&\forall\dbv\in\cV(\mesh), \label{guma}\\
   &b(\dbu,\bv) &&= 0\quad
   &&\forall\dbu\in\cU(\mesh). \label{gumb}
\end{alignat}
\end{subequations}
System \eqref{gu} is well posed by Theorem~\ref{thm_gu} and therefore, mixed formulation
\eqref{gum} is well posed as well.

An abstract discretization is formulated in the canonical way.
We select finite-dimensional subspaces $\cV_h(\mesh)\subset\cV(\mesh)$, $\cU_h(\mesh)\subset\cU(\mesh)$, and
assume the existence of a Fortin operator $\opF:\;\cV(\mesh)\to\cV_h(\mesh)$,
\begin{subequations} \label{Fgum}
\begin{alignat}{2}
   &b(\dbu,\bv-\opF\bv)=0\
   &&\forall \dbu\in \cU_h(\mesh),\ \forall \bv\in\cV(\mesh),\\
   \exists \cF>0:\quad & \|\opF \bv\|_{\cV(\mesh)}\le \cF \|\bv\|_{\cV(\mesh)}\
   &&\forall \bv\in \cV(\mesh),
\end{alignat}
\end{subequations}
to conclude the well-posedness and quasi-optimal convergence of the
\emph{mixed general ultraweak (DPG) method}:
\emph{Find $\bv_h\in\cV_h(\mesh)$ and $\bu_h\in\cU_h(\mesh)$ such that}
\begin{subequations} \label{gumh}
\begin{alignat}{3}
   &\ip{\bv_h}{\dbv}_{\cV(\mesh)} + b(\bu_h,\dbv) &&= L(\dbv)\quad
   &&\forall\dbv\in\cV_h(\mesh), \label{gumha}\\
   &b(\dbu,\bv_h) &&= 0\quad
   &&\forall\dbu\in\cU_h(\mesh). \label{gumhb}
\end{alignat}
\end{subequations}

\begin{theorem} \label{thm_gumh}
Under the conditions of Theorem~\ref{thm_gu}, and assuming \eqref{Fgum} with a constant
$\cF$ independent of $\mesh$ and the discrete subspaces,
scheme \eqref{gumh} has a unique solution
$(\bv_h,\bu_h)$ with components $\bu_h=(u_h,w_h,\phi_h,\psi_h)$.
It satisfies
\[
   \|\bu-\bu_h\|_{\cU(\mesh)}^2
   = \|u-u_h\|^2 + \|w-w_h\|^2 + \|\phi-\phi_h\|_{A,\cS}^2 + \|\psi-\psi_h\|_{A^*,\cS}^2
   \le
   C \|\bu-\bw\|_{\cU(\mesh)}^2
\]
for any $\bw\in \cU_h(\mesh)$ with a constant $C>0$ that is independent of $f$, $\mesh$,
and the discrete subspaces.
\end{theorem}

\begin{proof}
By the well-posedness of \eqref{gum} and the existence of a Fortin operator,
scheme \eqref{gumh} is well posed.
Equation~\eqref{gumha} means that $\bv_h$ is the Riesz representation
of the residual $L-b(\bu_h,\cdot)\in \cV_h(\mesh)^*$,
and \eqref{gumhb} implies that
\begin{align*}
   \bu_h=\mathrm{arg\,min}_{\bw\in\cU_h(\mesh)} \|b(\bw,\cdot)-L\|_{\cV_h(\mesh)^*},
\end{align*}
leading to the claimed error estimate. For details we refer to
\cite[Propositions~2.2,~2.3]{BroersenS_14_RPG}, see also
\cite[Theorem~2.1]{GopalakrishnanQ_14_APD}.
\end{proof}

\section{Applications to the Kirchhoff--Love model} \label{sec_KL}

The Kirchhoff--Love model of plate bending reads
\begin{subequations}\label{KLove}
\begin{align}
  \div\Div \bM = f,\quad \bM = \cC D^2 u \quad &\text{in }\Omega, \label{KL1}\\
  u=\partial_n u =0 \quad &\text{on }\Gamma:=\partial\Omega. \label{KL2}
\end{align}
\end{subequations}
Here, $u:\;\Omega\to\R$ is the deflection of a plate with mid-surface $\Omega\subset\R^2$,
$\bM:\;\Omega\to \SS$ is the tensor of bending moments, and
$f:\;\Omega\to\R$ represents the external vertical load. Tensor $\cC$ is symmetric,
positive definite and represents the rigidity of the plate.
Expressions $\Div\bM$, $D^2 u$, and $\partial_n u$ denote the row-wise divergence of $\bM$,
the Hessian of $u$, and the exterior normal derivative of $u$, respectively.
In principle $\Omega$ can be a bounded, simply connected Lipschitz domain.
For ease of discrete analysis we assume that it is a Lipschitz polygon.

\subsection{Notation of spaces, operators and mesh data} \label{sec_notation}

Given a subdomain $\omega$ of $\Rt$ or a line segment,
we use the canonical Lebesgue and Sobolev spaces $L_2(\omega)$ and $H^s(\omega)$
($0<s\le 2$) of scalar functions,
and need the corresponding spaces $L_2(\omega;U)$, $H^s(\omega;U)$ of functions
with values in $U\in\{\R,\Rt,\SS\}$. The $L_2(\omega;U)$ duality and norm
will be generically denoted by $\vdual{\cdot}{\cdot}_\omega$ and $\|\cdot\|_\omega$.
We use canonical Sobolev norms in $H^s(\omega;U)$, indicated by $\|\cdot\|_{s,\omega}$,
cf.~\cite{Adams},
except for $s=2$ when $\|\cdot\|_{2,\omega}^2:=\|\cdot\|_\omega^2+\|D^2\cdot\|_\omega^2$
with Hessian $D^2$.
We need the space of bending moments
\begin{align*}
   H(\dDiv,\omega;\SS)
   &:= \{\bM\in L_2(\omega;\SS);\; \div\Div\bM\in L_2(\omega;\R)\}
\end{align*}
with (squared) norm
$\|\cdot\|_{\dDiv,\omega}^2:=\|\cdot\|_\omega^2+\|\div\Div\cdot\|_\omega^2$.
Furthermore, $H^1_0(\omega;U)$ and $H^2_0(\omega;U)$ are the respective subspaces with
homogeneous traces on the boundary $\partial\omega$
(of course, including the normal derivative for $H^2$).
We generally drop index $\omega$ when $\omega=\Omega$.

We consider a regular (family of) mesh(es) $\mesh$
consisting of shape-regular triangles $\el$ covering $\Omega$,
$\cup_{\el\in\mesh}\overline{\el}=\overline{\Omega}$.
The set of (open) edges of $\mesh$ is denoted by $\cE$,
$\cE(\Omega):=\{E\in\cE;\; E\subset\Omega\}$, and $\cE(\Gamma):=\{E\in\cE;\; E\subset\Gamma\}$.
The set of edges of $\el\in\mesh$ is $\cE(\el)$. There are corresponding sets
of vertices $\cN$, vertices $\cN(\el)$ of elements $\el\in\mesh$,
vertices $\cN(\Omega)$ interior to $\Omega$, and vertices $\cN(\Gamma)$ on $\Gamma$.
We also need the exterior unit normal and tangential vectors,
defined almost everywhere on $\partial\el$, generically denoted by $\bn$ and $\bt$,
respectively.
Occasionally we use barycentric coordinates to specify basis functions.
For every element $\el\in\mesh$, we generically denote its coordinates by $\lambda_j$ ($j=1,2,3$),
corresponding to vertices $x_j$ and opposite edges $E_j$, numbered modulo $3$. That is, e.g.,
$\{\lambda_j,\lambda_{j+1},\lambda_{j+2}\}=\{\lambda_1,\lambda_2,\lambda_3\}$ for any
integer $j$.

Mesh $\mesh$ induces product spaces denoted as before, but replacing
$\omega$ with $\mesh$, e.g., $H^1(\mesh):=\Pi_{\el\in\mesh} H^1(\el)$.
We will identify elements of product spaces with piecewise defined functions
on $\Omega$, e.g., $v\in H^1(\Omega)$ is identified with $v\in H^1(\mesh)$, and
$v\in H^1(\mesh)$ corresponds to an element $v\in L_2(\Omega)$.
Furthermore, adding index $\mesh$ to a differential operator means that
it is considered to be a piecewise operator, e.g., $\div_\mesh$ is the $\mesh$-piecewise
divergence operator. The $L_2(\mesh)$ duality is denoted by $\vdual{\cdot}{\cdot}_\mesh$.

Throughout this section we select, consistent with the notation just introduced,
\[
   U=\SS,\quad A=D^2,\quad A^*=\div\Div,\quad
   \|\cdot\|_{2,\mesh} = \|\cdot\|_{A,\mesh},\quad
   \|\cdot\|_{\dDiv,\mesh} = \|\cdot\|_{A^*,\mesh}
\]
and refer to the fixed canonical $H^2$-trace operator as
\begin{align} \label{trtwo}
   &\trtwo:=\trAS:\; H(A)=H^2(\Omega) \to H(\dDiv,\mesh;\SS)^*
\end{align}
(recall definition \eqref{trAS} of $\trAS$).
Up to a sign change, $\trtwo$ is the trace operator $\mathrm{tr}^\mathrm{Ggrad}$
from \cite{FuehrerHN_19_UFK}.

For discretizations we need the space $P^s(\el)$ that consists of polynomials
of degree less than or equal to $s\in\N_0$ (the non-negative integers) on $\el\in\mesh$.
According to our notation, $P^s(\mesh)$ is the piecewise polynomial space
(without continuity requirement).
We also need the space $P^s_\mathrm{hom}(\el)$ of homogeneous polynomials of degree $s$.
The $L_2(\Omega)$-projection operator onto $P^s(\mesh)$ is denoted as $\Pi^s_\mesh$.
The space of continuous piecewise polynomials is $P^{s,c}(\mesh):=P^s(\mesh)\cap H^1(\Omega)$.

\subsection{Primal hybrid formulation} \label{sec_KL_p1}

The following formulation gives the variational framework for a discretization that
Brezzi and Fortin call a primal hybrid method,
cf.~\cite[IV.1.3]{BrezziF_91_MHF}.
We select
\begin{align*}
   \sH(A,\mesh):=H(A,\mesh)=H^{2}(\mesh)
\end{align*}
and denote
\[
   \trdDiv:=\trAsS,\quad
   H^{-3/2,-1/2}(\cS) := H(A^*,\cS)=\trdDiv\bigl(H(\dDiv,\Omega;\SS)\bigr),\quad
   \|\cdot\|_{-3/2,-1/2,\cS}:=\|\cdot\|_{A^*,\cS}
\]
(recall definition \eqref{trAsS} of $\trAsS$).
In this case, with $\sH(A,\mesh)=H(A,\mesh)$, $\trAsS=\trdDiv$ is the canonical trace
operator introduced in \cite{FuehrerHN_19_UFK}.

\begin{remark}
It turns out that trace operator $\trdDiv$ gives rise to different components
that are not independent and cannot be localized in general.
For instance, tensors $\bM\in H(\dDiv,\Omega;\SS)$ generally do not satisfy
$\Div\bM\in L_2(\Omega;\Rt)$ and thus do not allow for two independent integrations
by parts for operator $A^*=\div\Div$, cf.~\cite[Remark~3.1]{FuehrerHN_19_UFK}.
For discretization, however, the unknown from $H^{-3/2,-1/2}(\cS)$ will be the trace
of a more regular tensor in which case such a localized decomposition is possible,
and needed to define discrete subspaces of $H^{-3/2,-1/2}(\cS)$.
This will be discussed below, cf.~\eqref{IP},~\eqref{trdd}.
\end{remark}

The generalized primal hybrid formulation \eqref{gp} becomes the
\emph{primal hybrid formulation} of the Kirchhoff--Love model:
\emph{Find $u\in H^{2}(\mesh)$ and $\boldeta\in H^{-3/2,-1/2}(\cS)$ such that}
\begin{subequations} \label{KL_p1}
\begin{alignat}{3}
   &\vdual{\cC D^2 u}{D^2\du}_\mesh + \dual{\boldeta}{\du}_\cS
   &&= \vdual{f}{\du}, \label{KL_p1a}\\
   & \dual{\dbeta}{u}_\cS &&= 0 \label{KL_p1b}
\end{alignat}
\end{subequations}
\emph{holds for any $\du\in H^{2}(\mesh)$ and $\dbeta\in H^{-3/2,-1/2}(\cS)$.}

\begin{theorem} \label{thm_KL_p1}
Let $f\in L_2(\Omega)$ be given.
Problem \eqref{KL_p1} is well posed. Its solution $(u,\boldeta)$ satisfies
\[
   \|u\|_{2,\mesh} + \|\boldeta\|_{-3/2,-1/2,\cS} \le C \|f\|
\]
with a constant $C$ that is independent of $f$ and $\mesh$.
Furthermore, $u\in H^2_0(\Omega)$, $\boldeta=\trdDiv(\cC D^2 u)$,
and $u$, $\bM:=\cC D^2 u$ solve \eqref{KLove}.
\end{theorem}

\begin{proof}
We note that $H_0(A)=H^2_0(\Omega)$, cf.~\cite[Proof of Proposition~3.8(i)]{FuehrerHN_19_UFK}.
The Poincar\'e--Friedrichs inequality holds by \cite[Lemma~3.3]{BabuskaP_90_PPH}.
An application of Theorem~\ref{thm_gp} proves the statements.
\end{proof}

\subsubsection*{Discretization}

In order to discretize trace space $H^{-3/2,-1/2}(\cS)$ we need to represent
trace operator $\trdDiv$ explicitly. Integration by parts shows that
$u\in H^2(\mesh)$ and sufficiently smooth tensor functions $\bM\in H(\dDiv,\mesh;\SS)$
satisfy
\begin{align} \label{IP}
   &\vdual{\div\Div\bM}{u}_\mesh - \vdual{\bM}{D^2 u}_\mesh
   = \sum_{\el\in\mesh}
     \dual{\bn\cdot\Div\bM}{u}_{\partial\el} - \dual{\bM\bn}{\grad u}_{\partial\el}
   \nonumber\\
   &= \sum_{\el\in\mesh}
          \Bigl(\dual{\bn\cdot\Div\bM+\partial_t(\bt\cdot\bM\bn)}{u}_{\partial\el}
         - \sum_{x\in\cN(\el)} [\bt\cdot\bM\bn]_{\partial\el}(x) u(x)
         - \dual{\bn\cdot\bM\bn}{\partial_n u}_{\partial\el}\Bigr).
\end{align}
Here, $[\bt\cdot\bM\bn]_{\partial\el}(x)$ denotes the jump at node $x\in\cN(\el)$
of the trace $\bt\cdot\bM\bn|_{\partial\el}$ from within $\el$ (in a certain orientation),
$\partial_n u:=\bn\cdot\grad u$, and
$\partial_t=\bt\cdot\grad$ is the tangential derivative on $\partial\el$ in positive orientation.
For details and a Sobolev space setting we refer to \cite[Section~3]{FuehrerHN_19_UFK}.

Relation \eqref{IP} gives rise to trace operator $\trdDiv$ by selecting tensors without
jumps $\bM\in H(\dDiv,\Omega;\SS)$. Trace $\trdDiv(\bM)$ has the following components:
\begin{subequations} \label{trdd}
\begin{alignat}{3}
   &\bn\cdot\Div\bM+\partial_t(\bt\cdot\bM\bn)|_E,\quad E\in\cE
   && \text{(effective shear-force(s))}, \label{trdd1}\\
   &\bn\cdot\bM\bn|_E,\quad E\in\cE
   && \text{(normal-normal traces)}, \label{trdd2}\\
   &[\bt\cdot\bM\bn]_{\partial\el}(x),\quad x\in\cN(\el),\ \el\in\mesh
   && \text{(corner forces)} \label{trdd3}\\
   \text{subject to}\quad
   &\sum_{\el\in\mesh:\; x\in\cN(\el)} [\bt\cdot\bM\bn]_{\partial\el}(x)=0,\quad
   x\in\cN(\Omega). \label{trdd4}
\end{alignat}
\end{subequations}
The lowest-order non-trivial discretization consists of edge-piecewise constants for
both the effective shear forces and normal-normal traces, and the corner forces
are point values subject to the constraints at interior vertices,
cf.~\cite[(6.5)]{FuehrerHN_19_UFK}. We represent this space
as $P^0(\cS)\times P^0(\cS)\times\R^{3|\mesh|}_\mathrm{constr}$ with
\[
   P^0(\cS) := \{\eta:\;\cS\to \R;\;
                 \eta|_E\ \text{is a univariate constant}\ \forall E\in\cE\},
\]
and where $\R^{3|\mesh|}_\mathrm{constr}\subset\R^{3|\mesh|}$
denotes the subspace of $3|\mesh|$ point values
\eqref{trdd3} subject to the constraints \eqref{trdd4}.
Interpreting trace space $H^{-3/2,-1/2}(\cS)$ as a space with three components
(effective shear forces, normal-normal traces, corner forces),
our discretization spaces are
\begin{align*}
   \sH_h(A,\mesh)&:= P^3(\mesh) \subset H^2(\mesh),\\
   H_h(A^*,\cS)  &:= P^0(\cS)\times P^0(\cS)\times\R^{3|\mesh|}_\mathrm{constr}
   \subset H^{-3/2,-1/2}(\cS).
\end{align*}
They have dimensions $10|\mesh|$ and $2|\cE|+3|\mesh|-|\cN(\Omega)|$, respectively.

The resulting primal hybrid method reads as follows.
\emph{Find $u_h\in P^3(\mesh)$ and
$\boldeta_h\in P^0(\cS)\times P^0(\cS)\times\R^{3|\mesh|}_\mathrm{constr}$ such that}
\begin{subequations} \label{KL_p1_h}
\begin{alignat}{3}
   &\vdual{\cC D^2 u_h}{D^2\du}_\mesh + \dual{\boldeta_h}{\du}_\cS
   &&= \vdual{f}{\du},\\
   & \dual{\dbeta}{u_h}_\cS &&= 0
\end{alignat}
\end{subequations}
\emph{holds for any $\du\in P^3(\mesh)$ and
$\dbeta\in P^0(\cS)\times P^0(\cS)\times\R^{3|\mesh|}_\mathrm{constr}$.}
It converges quasi-optimally.

\begin{theorem} \label{thm_KL_p1_h}
Let $f\in L_2(\Omega)$ be given. System \eqref{KL_p1_h} is well posed.
Its solution $(u_h,\boldeta_h)$ satisfies
\[
   \|u-u_h\|_{2,\mesh} + \|\boldeta-\boldeta_h\|_{-3/2,-1/2,\cS}
   \le C \Bigl( \|u-v\|_{2,\mesh} + \|\boldeta-\bpsi\|_{-3/2,-1/2,\cS}\Bigr)
\]
for any $v\in P^3(\mesh)$ and
$\bpsi\in P^0(\cS)\times P^0(\cS)\times\R^{3|\mesh|}_\mathrm{constr}$.
Here, $(u,\boldeta)$ is the solution of \eqref{KL_p1} and $C>0$ is independent of
$\mesh$, $v$ and $\bpsi$.
\end{theorem}

\begin{proof}
The proof is given in abstract form by Theorem~\ref{thm_gph}, once the conditions
made there are verified. The Poincar\'e--Friedrichs inequality holds for $H^2_0(\Omega)$ so that
we only have to check the existence of a Fortin operator $\opF$ satisfying \eqref{Fgp}
and the validity of \eqref{PFh}.
A Fortin operator is given in \cite[Lemma~11]{FuehrerH_19_FDD}, see $\Pi^\mathrm{Ggrad}$ there,
and \eqref{PFh} follows from \cite[Theorem~3.1]{WangX_13_MFE}.
\end{proof}

\begin{remark}
We note that our hybrid framework allows for the trivial approximation of
effective shear force traces, that is, selecting
$H_h(A^*,\cS):=\{0\}\times P^0(\cS)\times\R^{3|\mesh|}_\mathrm{constr}$.
Then the deflection $u$ can be approximated by piecewise quadratic elements
$u_h\in P^2(\mesh)$ and \eqref{KL_p1_h} turns out to be a hybridization
of the Morley element~\cite{Morley_68_TEE}.
\end{remark}

\subsection{Nodal-continuous primal hybrid formulation} \label{sec_KL_p2}

We present a formulation that gives the variational framework for a Morley-type element,
cf.~\cite{Morley_68_TEE}. We select
\[
   \sH(A,\mesh):=H^{2,\cN}_0(\mesh)
   := \{v\in H^{2}(\mesh);\; v\ \text{is continuous at }x\in\cN(\Omega)\
                                \text{and vanishes at }x\in \cN(\Gamma)\}
\]
and denote
$\trdDivM:=\trAsS$ and
\begin{align} \label{tracenorm_p2}
   &\HtrdDivM(\cS) := H(A^*,\cS)=\trdDivM\bigl(H(\dDiv,\Omega;\SS)\bigr),\quad
   \|\cdot\|_{-3/2,-1/2,J0,\cS}:=\|\cdot\|_{A^*,\cS}.
\end{align}
Index notation ``$J0$'' (``jump zero'') refers to the fact that tangential-normal jumps of bending
moments do not appear, see~\eqref{IP_p2} below.

The generalized primal hybrid formulation \eqref{gp} becomes the
\emph{nodal-continuous primal hybrid formulation} of the Kirchhoff--Love model:
\emph{Find $u\in H^{2,\cN}_0(\mesh)$ and
$\boldeta\in \HtrdDivM(\cS)$ such that}
\begin{subequations} \label{KL_p2}
\begin{alignat}{3}
   &\vdual{\cC D^2 u}{D^2\du}_\mesh + \dual{\boldeta}{\du}_\cS
   &&= \vdual{f}{\du},\\
   & \dual{\dbeta}{u}_\cS &&= 0
\end{alignat}
\end{subequations}
\emph{holds for any $\du\in H^{2,\cN}_0(\mesh)$ and
$\dbeta\in \HtrdDivM(\cS)$.}

\begin{theorem} \label{thm_KL_p2}
Let $f\in L_2(\Omega)$ be given.
Problem \eqref{KL_p2} is well posed. Its solution $(u,\boldeta)$ satisfies
\[
   \|u\|_{2,\mesh} + \|\boldeta\|_{-3/2,-1/2,J0,\cS} \le C \|f\|
\]
with a constant $C$ that is independent of $f$ and $\mesh$.
Furthermore, $u\in H^2_0(\Omega)$, $\boldeta=\trdDivM(\cC D^2 u)$,
and $u$, $\bM:=\cC D^2 u$ solve \eqref{KLove}.
\end{theorem}

\begin{proof}
The proof is identical to the proof of Theorem~\ref{thm_KL_p1}.
\end{proof}

\subsubsection*{Discretization}

To discretize trace space $\HtrdDivM(\cS)$, we consider $u\in H^{2,\cN}_0(\mesh)$
and proceed as in \eqref{IP} to find that a sufficiently smooth (piecewise polynomial) tensor
$\bM\in H(\dDiv,\Omega;\SS)$ satisfies
\begin{equation} \label{IP_p2}
   \dual{\trdDivM(\bM)}{u}_\cS
   = \sum_{\el\in\mesh}
   \dual{\bn\cdot\Div\bM-\partial_t(\bt\cdot\bM\bn)}{u}_{\partial\el}
   -\dual{\bn\cdot\bM\bn}{\partial_n u}_{\partial\el}
   \quad\forall u\in H^{2,\cN}_0(\mesh).
\end{equation}
Here, we used that
\[
   \sum_{\el\in\mesh:\;x\in\cV(\el)} [\bt\cdot\bM\bn]_{\partial\el}(x)=0
   \quad\forall x\in\cN(\Omega)
\]
in distributional sense (testing with $C^\infty$-functions
with support in a neighborhood of $x$), see \cite[Proposition~3.6]{FuehrerHN_19_UFK}.
We conclude that trace $\trdDivM(\bM)$ has two scalar components living on the skeleton,
the effective shear force $\bn\cdot\Div\bM+\partial_t(\bt\cdot\bM\bn)$ \eqref{trdd1}
and the normal-normal trace $\bn\cdot\bM\bn$ \eqref{trdd2}.
Corner forces $[\bt\cdot\bM\bn]_{\partial\el}(x)$ \eqref{trdd3} do not appear.
We approximate both trace components by edge-piecewise constant functions,
and use nodal-continuous piecewise cubic polynomials to discretize $H^{2,\cN}_0(\mesh)$.
In our abstract notation, the approximation spaces are
\[
   \sH_h(A,\mesh):= \Xptwo(\mesh) :=  P^{3}(\mesh)\cap H^{2,\cN}_0(\mesh),\quad
   H_h(A^*,\cS)  := P^0(\cS)\times P^0(\cS).
\]
The dimensions of $\Xptwo(\mesh)$ and $P^0(\cS)\times P^0(\cS)$ are
$|\cN(\Omega)|+7|\mesh|$ and $2|\cE|$, respectively.

The resulting nodal-continuous primal hybrid method reads as follows.
\emph{Find $u_h\in \Xptwo(\mesh)$ and
$\boldeta_h\in P^0(\cS)\times P^0(\cS)$ such that}
\begin{subequations} \label{KL_p2_h}
\begin{alignat}{3}
   &\vdual{\cC D^2 u_h}{D^2\du}_\mesh + \dual{\boldeta_h}{\du}_\cS
   &&= \vdual{f}{\du},\\
   & \dual{\dbeta}{u_h}_\cS &&= 0
\end{alignat}
\end{subequations}
\emph{holds for any $\du\in \Xptwo(\mesh)$ and
$\dbeta\in P^0(\cS)\times P^0(\cS)$.}

In the numerical section we will refer to the shear-force and normal-normal trace
components of $\boldeta_h$ as $\eta^\mathit{sf}_h\in P^0(\cS)$ and
$\eta^\mathit{nn}_h\in P^0(\cS)$, respectively.

\begin{remark} \label{rem_KL_p2_h}
We note that scheme \eqref{KL_p1_h} can be interpreted as a hybridization of
scheme \eqref{KL_p2_h} or, vice versa, scheme \eqref{KL_p2_h} as a reduction
of \eqref{KL_p1_h} by elimination of vertex discontinuities of $u_h$.
In fact, the degrees stemming from the component
$\R^{3|\mesh|}_\mathrm{constr}$ in \eqref{KL_p1_h} fix the vertex values of $u_h$.
\end{remark}

Scheme \eqref{KL_p2_h} converges quasi-optimally.

\begin{theorem} \label{thm_KL_p2_h}
Let $f\in L_2(\Omega)$ be given. System \eqref{KL_p2_h} is well posed.
Its solution $(u_h,\boldeta_h)$ satisfies
\[
   \|u-u_h\|_{2,\mesh} + \|\boldeta-\boldeta_h\|_{-3/2,-1/2,J0,\cS}
   \le C \Bigl( \|u-v\|_{2,\mesh} + \|\boldeta-\bpsi\|_{-3/2,-1/2,J0,\cS}\Bigr)
\]
for any $v\in \Xptwo(\mesh)$ and $\bpsi\in P^0(\cS)\times P^0(\cS)$.
Here, $(u,\boldeta)$ is the solution of \eqref{KL_p2} and $C>0$ is independent of
$\mesh$, $v$ and $\bpsi$.
\end{theorem}

According to Remark~\ref{rem_KL_p2_h}, this theorem follows from Theorem~\ref{thm_KL_p1_h}
as a special case.


\subsection{Continuous primal hybrid formulation} \label{sec_KL_p3}

The next formulation gives rise to a scheme that uses continuous approximations
of deflections, similarly to the Zienkiewicz triangular element~\cite{BazeleyCIZ_65_TEP},
the non-conforming approach in \cite{BabuskaZ_73_NEF},
$C^0$-interior penalty methods \cite{EngelGHLMT_02_CDF,BrennerS_05_CIP},
though without stabilization, and the HHO method in  \cite{DongE_24_CHH}.

We select
\[
   \sH(A,\mesh):=H^{2,1}_0(\mesh):=H^{2}(\mesh)\cap H^1_0(\Omega)
   \subsetneq H(A,\mesh)=H^2(\mesh)
\]
and denote
\begin{align*}
   \trdnn:=\trAsS,\quad
   H^{-1/2}(\cS) := H(A^*,\cS)=\trdnn\bigl(H(\dDiv,\Omega;\SS)\bigr),\quad
   \|\cdot\|_{-1/2,\cS}:=\|\cdot\|_{A^*,\cS}.
\end{align*}
Notation $\trdnn$ is suggested by the fact that the duality with
test functions that are continuous across $\cS$ gives rise to the normal-normal traces
on $\cS$, see~\eqref{IP_p3} below.
For localized traces in this context with Banach spaces see also \cite[\S 2.4]{DongE_24_CHH}.

The generalized primal hybrid formulation \eqref{gp} becomes the
\emph{continuous primal hybrid formulation} of the Kirchhoff--Love model:
\emph{Find $u\in H^{2,1}_0(\mesh)$ and $\eta\in H^{-1/2}(\cS)$ such that}
\begin{subequations} \label{KL_p3}
\begin{alignat}{3}
   &\vdual{\cC D^2 u}{D^2\du}_\mesh + \dual{\eta}{\du}_\cS
   &&= \vdual{f}{\du},\\
   & \dual{\deta}{u}_\cS &&= 0
\end{alignat}
\end{subequations}
\emph{holds for any $\du\in H^{2,1}_0(\mesh)$ and $\deta\in H^{-1/2}(\cS)$.}

\begin{theorem} \label{thm_KL_p3}
Let $f\in L_2(\Omega)$ be given.
Problem \eqref{KL_p3} is well posed. Its solution $(u,\eta)$ satisfies
\[
   \|u\|_{2,\mesh} + \|\eta\|_{-1/2,\cS} \le C \|f\|
\]
with a constant $C$ that is independent of $f$ and $\mesh$.
Furthermore, $u\in H^2_0(\Omega)$, $\eta=\trdnn(\cC D^2 u)$,
and $u$, $\bM:=\cC D^2 u$ solve \eqref{KLove}.
\end{theorem}

\begin{proof}
The proof is identical to the proof of Theorem~\ref{thm_KL_p1}.
\end{proof}

\subsubsection*{Discretization}

In order to discretize trace space $H^{-1/2}(\cS)$ we need an explicit representation
of trace operator $\trdnn$ for sufficiently smooth (piecewise polynomial) tensors
$\bM\in H(\dDiv,\Omega;\SS)$. Proceeding as in \eqref{IP}, we find that such a tensor $\bM$
satisfies
\begin{equation} \label{IP_p3}
   \dual{\trdnn(\bM)}{u}_\cS
   = -\sum_{\el\in\mesh} \dual{\bn\cdot\bM\bn}{\partial_n u}_{\partial\el}
   \quad \forall u\in H^{2,1}_0(\mesh).
\end{equation}
Therefore, $\trdnn(\bM)|_E$ can be identified with $-\bn\cdot\bM\bn|_E$ on edges $E\in\cE$.
We approximate these normal-normal traces by edge-piecewise constant functions,
and use standard continuous, piecewise cubic polynomials, enriched with element-bubble
functions of degree $4$, to discretize $H^{2,1}_0(\mesh)$.
Specifically, we define the following spaces,
\begin{alignat*}{2}
   &P^4_b(\el) := P^4(\el)\cap H^1_0(\el),\quad
   &&\Xpthreeloc(\el) := P^3(\el) + P^4_b(\el)\quad (\el\in\mesh),\\
   &P^{3,c}_0(\mesh) := P^3(\mesh)\cap H^1_0(\Omega),\quad
   &&\Xpthree(\mesh) := P^{3,c}_0(\mesh) + P^4_b(\mesh)
\end{alignat*}
and, in our abstract notation, select the approximation spaces
\[
   \sH_h(A,\mesh):= \Xpthree(\mesh) \subset H^{2,1}_0(\mesh),\quad
   H_h(A^*,\cS)  := P^0(\cS) \subset H^{-1/2}(\cS).
\]
Here, $P^0(\cS)$ is the space of edge-piecewise constant functions on $\cS$ introduced previously.
The dimensions of $\Xpthree(\mesh)$ and $P^0(\cS)$ are $|\cN(\Omega)|+2|\cE(\Omega)|+3|\mesh|$
and $|\cE|$, respectively (see degrees of freedom \eqref{dofloc_p3} and note that
the continuity of $v\in\Xpthree(\mesh)$ requires unique vertex values $v(x)$ and edge moments
$\dual{v}{\phi}_E$).

The resulting continuous primal hybrid method reads as follows.
\emph{Find $u_h\in \Xpthree(\mesh)$ and $\eta_h\in P^0(\cS)$ such that}
\begin{subequations} \label{KL_p3_h}
\begin{alignat}{3}
   &\vdual{\cC D^2 u_h}{D^2\du}_\mesh + \dual{\eta_h}{\du}_\cS
   &&= \vdual{f}{\du},\\
   & \dual{\deta}{u_h}_\cS &&= 0
\end{alignat}
\end{subequations}
\emph{holds for any $\du\in \Xpthree(\mesh)$ and $\deta\in P^0(\cS)$.}
It converges quasi-optimally.

\begin{theorem} \label{thm_KL_p3_h}
Let $f\in L_2(\Omega)$ be given. System \eqref{KL_p3_h} is well posed.
Its solution $(u_h,\eta_h)$ satisfies
\[
   \|u-u_h\|_{2,\mesh} + \|\eta-\eta_h\|_{-1/2,\cS}
   \le C \Bigl( \|u-v\|_{2,\mesh} + \|\eta-\psi\|_{-1/2,\cS}\Bigr)
\]
for any $v\in \Xpthree(\mesh)$ and $\psi\in P^0(\cS)$.
Here, $(u,\eta)$ is the solution of \eqref{KL_p3} and $C>0$ is independent of
$\mesh$, $v$ and $\psi$.
\end{theorem}

A proof of this theorem is given by Theorem~\ref{thm_gph}.
As before, relation \eqref{PFh} follows from \cite[Theorem~3.1]{WangX_13_MFE} so that we
only have to verify the existence of a Fortin operator, denoted by $\opFpthree$.
This needs some preparation.  For a definition of $\opFpthree$ and the verification
of Fortin properties, see \eqref{def_F_p3} and Lemma~\ref{la_F_p3} below.

\begin{lemma} \label{la_dofloc_p3}
Given $\el\in\mesh$, the local space $\Xpthreeloc(\el)$ has dimension $12$.
A function $v\in \Xpthreeloc(\el)$ has the degrees of freedom
\begin{align} \label{dofloc_p3}
   v(x),\quad \dual{v}{\phi}_E,\quad \dual{\partial_n v}{1}_E
\end{align}
for $x\in\cV(\el)$, $\phi\in P^1(E)$, $E\in\cE(\el)$.
\end{lemma}

\begin{proof}
For $\el\in\mesh$ we use the notation $\lambda_j$ for barycentric coordinates,
vertices $x_j$, and edges $E_j$.  Let us denote $\psi_0:=\lambda_1\lambda_2\lambda_3$,
the lowest-order polynomial bubble function.
Space $P^4_b(\el)$ is spanned by $\{\phi\psi_0;\; \phi\in P^1(\el)\}$ with
basis $\{\psi_j:=\lambda_j\psi_0;\; j=1,2,3\}$ and therefore has dimension $3$.
Since $P^3(\el)\cap P^4_b(\el)$ is generated by $\psi_0$,
we conclude that $\Xpthreeloc(\el)$ has dimension $10+3-1=12$,
equal to the number of claimed degrees of freedom. It is enough to show their injectivity.

Let $v\in\Xpthree(\el)$ be given with vanishing degrees of freedom \eqref{dofloc_p3}.
The vanishing of the vertex values means that
$v|_{E_j}\in (\lambda_{j+1}\lambda_{j+2})|_{E_j}P^1(E_j)$, and
its orthogonality to $P^1(E_j)$ implies that $v|_{E_j}=0$, $j=1,2,3$,
thus $v\in P^4_b(\el)$. Therefore, function $v$ has a representation
$v=\sum_{j=1}^3 c_j\psi_j$. We calculate
\begin{align*}
   \grad v
   = 2\lambda_1\lambda_2\lambda_3 \sum_{j=1}^3 c_j\grad\lambda_j
   + \sum_{j=1}^3 c_j \lambda_j^2
            \Bigl(\lambda_{j+1}\grad\lambda_{j+2}+\lambda_{j+2}\grad\lambda_{j+1}\Bigr),
\end{align*}
that is, on edge $E_k$,
\begin{align*}
   \grad v   = c_{k+1}\lambda_{k+1}^2\lambda_{k+2}\grad\lambda_{k}
             + c_{k+2}\lambda_{k+2}^2\lambda_{k+1}\grad\lambda_{k}
             = \Bigl(c_{k+1}\lambda_{k+1}+c_{k+2}\lambda_{k+2}\Bigr)
               \lambda_{k+1}\lambda_{k+2}\grad \lambda_k. 
\end{align*}
Moments $\dual{\partial_n v}{1}_{E_k}$ vanish iff
$c_{k+1}\lambda_{k+1}+c_{k+2}\lambda_{k+2}$ vanishes at the midpoint of $E_k$,
that is, iff $c_{k}+c_{k+1}=0$, $k=1,2,3$, with only solution $c_1=c_2=c_3=0$, $v=0$.
\end{proof}

We prove the existence of a Fortin operator
$\opFpthree:\;H^{2,1}_0(\mesh)\to \Xpthree(\mesh)$ in two steps.

\begin{lemma} \label{la_F0_p3}
There is an operator $\opFpthreea:\;H^{2,1}_0(\mesh)\to P^4_b(\mesh)$ that satisfies
\be \label{Fa1_p3}
   \dual{\deta}{\opFpthreea u}_\cS
   = \dual{\deta}{u}_\cS\quad\forall \deta\in P^0(\cS),\ u\in H^{2,1}_0(\mesh)
\ee
and
\be \label{Fa2_p3}
   \|\opFpthreea u\| + \|h_\mesh^2 D^2\opFpthreea u\|_\mesh \lesssim \|u\| + \|h_\mesh^2 D^2 u\|_\mesh
   \quad\forall u\in H^{2,1}_0(\mesh).
\ee
\end{lemma}

\begin{proof}
By Lemma~\ref{la_dofloc_p3}, for every $u\in H^{2,1}_0(\Omega)$ there are unique
polynomials $q_\el\in P^4_b(\el)$ with the respective degrees of freedom
(edge moments of normal derivative) from $u|_\el$, $\el\in\mesh$.
They give rise to a function $q\in H^1_0(\Omega)$ ($q|_\el:=q_\el$ $\forall\el\in\mesh$)
whose normal derivatives on edges have jumps (values on edges $\subset\Gamma$)
with mean-value zero, that is, $\opFpthreea u:=q$ satisfies \eqref{Fa1_p3}.
By the boundedness of the associated degrees of freedom
(in $H^2$ on a reference element) and scaling properties, one establishes \eqref{Fa2_p3}.
\end{proof}

We define
\be \label{def_F_p3}
   \opFpthree:\; H^{2,1}_0(\mesh)\ni u\mapsto \opFpthreea\bigl(u-\Ipol^1_\mesh u) + \Ipol^1_\mesh u
\ee
where $\Ipol^1_\mesh:\; H^{2}(\mesh)\to P^1(\mesh)$ is the element-piecewise nodal
interpolation operator.
Of course, it maps $H^{2,1}_0(\mesh)\to P^{1,c}_0(\mesh)$.

\begin{lemma} \label{la_F_p3}
Operator $\opFpthree$ defined in \eqref{def_F_p3} is a mapping
$\opFpthree:\;H^{2,1}_0(\mesh)\to \Xpthree(\mesh)$ and satisfies
\be \label{F1_p3}
   \dual{\deta}{\opFpthree u}_\cS
   = \dual{\deta}{u}\quad\forall \deta\in P^0(\cS),\ u\in H^{2,1}_0(\mesh),
\ee
\be \label{F2_p3}
   \|\opFpthree u\| + \|D^2\opFpthree u\|_\mesh \lesssim \|u\| + \|D^2 u\|_\mesh
   \quad\forall u\in H^{2,1}_0(\mesh).
\ee
\end{lemma}

\begin{proof}
Let $u\in H^{2,1}_0(\mesh)$ be given. By construction,
$\opFpthree u\in P^{1,c}_0(\mesh)\oplus P^4_b(\mesh)$, and $\Ipol^1_\mesh u\in P^{1,c}_0(\mesh)$.
Therefore, relation \eqref{Fa1_p3} implies \eqref{F1_p3}:
\[
   \dual{\deta}{\opFpthree u}_\cS
   = \dual{\deta}{\opFpthreea (u-\Ipol^1_\mesh u)}_\cS + \dual{\deta}{\Ipol^1_\mesh u}_\cS
   = \dual{\deta}{u}_\cS\quad\forall \deta\in P^0(\cS).
\]
Property $D^2(\Ipol^1_\mesh u)|_\el=0$ ($\el\in\mesh$), bound \eqref{Fa2_p3}
and the approximation property $\|u-\Ipol^1_\mesh u\|\lesssim \|h_\mesh^2 D^2 u\|_\mesh$
show that $\opFpthree u$ satisfies
\begin{align*}
   \|D^2\opFpthree u\|_\mesh = \|D^2\opFpthreea(u-\Ipol^1_\mesh u)\|
   \lesssim \|h_\mesh^{-2} (u-\Ipol^1_\mesh u)\| + \|D^2 (u-\Ipol^1_\mesh u)\|_\mesh
   \lesssim \|D^2 u\|_\mesh.
\end{align*}
The $L_2$-estimate in \eqref{F2_p3} follows with \eqref{Fa2_p3} and scaling arguments to bound
$\|\Ipol^1_\mesh u\|\lesssim \|u\| + \|h_\mesh^2 D^2 u\|_\mesh$:
\begin{align*}
   \|\opFpthree u\| \le \|\opFpthreea(u-\Ipol^1_\mesh u)\| + \|\Ipol^1_\mesh u\|
   \lesssim \|u-\Ipol^1_\mesh u\| + \|h_\mesh^2 D^2 u\| + \|\Ipol^1_\mesh u\|
   \lesssim \|u\| + \|D^2 u\|_\mesh.
\end{align*}
This finishes the proof.
\end{proof}

\subsection{Mixed hybrid formulation} \label{sec_KL_d1}

The following formulation gives rise to the so-called
assumed stresses hybrid method by Pian and Tong \cite{PianT_69_BFE},
cf.~the analysis by Brezzi and Marini \cite{BrezziM_75_NSP}.
We select
\[
   \sH(A^*,\mesh):=H(A^*,\mesh)=H(\dDiv,\mesh;\SS)
\]
and denote
\begin{align*}
   H_0^{3/2,1/2}(\cS) &:=H_0(A,\cS)= \trtwo\bigl(H^2_0(\Omega)\bigr),\quad
   \|\cdot\|_{3/2,1/2,\cS} := \|\cdot\|_{A,\cS}.
\end{align*}
In this case, with $\sH(A^*,\mesh)=H(A^*,\mesh)$, $\trAS=\trtwo$ defined before in \eqref{trtwo}.

The generalized mixed hybrid formulation \eqref{gm} becomes the
\emph{mixed hybrid formulation} of the Kirchhoff--Love model:
\emph{Find $\bM\in H(\dDiv,\mesh;\SS)$, $u\in L_2(\Omega)$,
and $\bpsi\in H_0^{3/2,1/2}(\cS)$ such that}
\begin{subequations} \label{KL_d1}
\begin{alignat}{3}
   &\vdual{\cCinv\bM}{\dbM} - \vdual{u}{\div\Div\dbM}_\mesh - \dual{\bpsi}{\dbM}_\cS
   &&= 0, \label{KL_d1a}\\
   &-\vdual{\div\Div\bM}{\du}_\mesh - \dual{\dbpsi}{\bM}_\cS
   &&= -\vdual{f}{\du} \label{KL_d1b}
\end{alignat}
\end{subequations}
\emph{holds for any $\dbM\in H(\dDiv,\mesh;\SS)$, $\du\in L_2(\Omega)$,
and $\dbpsi\in H_0^{3/2,1/2}(\cS)$.}
It is an extended form of formulation \cite[(1.23)]{BrezziM_75_NSP}
without an extension of right-hand side function $f$ to an element of
$H(\dDiv,\mesh;\SS)$.

\begin{theorem} \label{thm_KL_d1}
Let $f\in L_2(\Omega)$ be given.
Problem \eqref{KL_d1} is well posed. Its solution $(\bM,u,\bpsi)$ satisfies
\[
   \|\bM\|_{\dDiv,\mesh} + \|u\| + \|\bpsi\|_{3/2,1/2,\cS} \le C \|f\|
\]
with a constant $C$ that is independent of $f$ and $\mesh$.
Furthermore, $u\in H^2_0(\Omega)$, $\bM=\cC D^2 u$, $\bpsi=\trtwo(u)$,
and $u$, $\bM$ solve \eqref{KLove}.
\end{theorem}

\begin{proof}
We have already seen that the Poincar\'e--Friedrichs inequality \eqref{PF} holds true.
Inf-sup property \eqref{infsup} reads
\[
   \sup_{\bQ\in H(\dDiv,\Omega;\SS)} \frac{\vdual{\div\Div\bQ}{v}}{\|\bQ\|_{\dDiv}}
   \ge \cis \|v\|\quad \forall v\in L_2(\Omega)
\]
and is satisfied by the surjectivity of $\div\Div:\;H(\dDiv,\Omega;\SS)\to L_2(\Omega)$.
In fact, given $g\in L_2(\Omega)$, there is a (unique) solution
$v_g\in H^2_0(\Omega)$ to $\div\Div D^2 v_g = \Delta^2 v_g=g$ in $\Omega$, and
$\bQ:=D^2 v_g\in H(\dDiv,\Omega;\SS)$ satisfies $\div\Div\bQ=g$.
Theorem~\ref{thm_gm} proves the statements.
\end{proof}

\subsubsection*{Discretization}

In order to discretize trace space $H_0^{3/2,1/2}(\cS)$ we need to represent trace
operator $\trtwo$ explicitly. This is dual to the setting of the primal hybrid formulation,
cf.~\eqref{IP}. By our definition of $\trtwo$ we have a sign change:
\begin{align} \label{IP_d1}
   &\dual{\trtwo(u)}{\bM}_\cS
   = \sum_{\el\in\mesh} \dual{\trtwoloc(u)}{\bM}_{\partial\el}
   := \sum_{\el\in\mesh} \vdual{D^2 u}{\bM}_\el -\vdual{u}{\div\Div\bM}_\el\nonumber\\
   &= \sum_{\el\in\mesh}
          \Bigl(\dual{\bn\cdot\bM\bn}{\partial_n u}_{\partial\el}
                -
                \dual{\bn\cdot\Div\bM+\partial_t(\bt\cdot\bM\bn)}{u}_{\partial\el}
         + \sum_{x\in\cN(\el)} [\bt\cdot\bM\bn]_{\partial\el}(x) u(x)\Bigr)
\end{align}
for $u\in H^2_0(\Omega)$ and a sufficiently piecewise-smooth tensor $\bM\in H(\dDiv,\mesh;\SS)$.
Here, we introduced the local trace operators $\trtwoloc$ for notational convenience.
Trace $\bpsi|_{\partial\el}=\trtwoloc(u)$ consists of the components
$u|_{\partial\el}$ and $\partial_n u|_{\partial\el}$.
We approximate $\bpsi$ by traces of the reduced Hsieh--Clough--Tocher (HCT) composite element
which are edge-piecewise cubic polynomials, cf.~\cite{Ciarlet_78_IEE,FuehrerHN_19_UFK},
\begin{align} \label{HCT}
   \HCT(\el) &:= 
      \{v\in H^2(\el);\; \Delta^2 v + v=0,\
        v|_E\in P^3(E),\ \partial_n v|_E\in P^1(E)\ \forall E\in\cE(\el)\},\\
   \HCT_0^{2}(\cS) &:= \trtwo\Bigl(\HCT(\mesh)\cap H^2_0(\Omega)\Bigr). \nonumber
\end{align}
Bending moments are approximated by a reduction of the $H(\dDiv;\SS)$ element from
\cite{FuehrerH_MKL}. Specifically, for $\el\in\mesh$ we use the Raviart--Thomas spaces
\[
   \RT^s(\el)=\bx P^s_\mathrm{hom}(\el)\oplus P^s(\el;\Rt)
   \quad (s\in\N_0)
\]
where $\bx:\;\R^2\ni(x_1,x_2)\mapsto (x_1,x_2)^\top$,
denote $\sym(\bQ):=(\bQ+\bQ^\top)/2$ and introduce
\begin{align*}
   \XdDiv(\el) &:= \mathrm{sym}\bigl(\RT^0(\el)\otimes\RT^1(\el)\bigr).
\end{align*}
We approximate bending moments by the $\XdDiv(\el)$-element reduced to constant normal-normal
edge traces,
\begin{align*}
   \XdDivnnc(\el) &:= \{\bM\in\XdDiv(\el);\; \bn\cdot\bM\bn|_E\in P^0(E),\ E\in\cE(\el)\}.
\end{align*}
In our abstract notation, the approximation spaces are
\[
   \sH_h(A^*,\mesh):= \XdDivnnc(\mesh),\quad
   H_h(\mesh) := P^1(\mesh),\quad
   H_h(A,\cS) := \HCT^{2}_0(\cS).
\]
Their respective dimensions are $12|\mesh|$, $3|\mesh|$, and $3|\cN(\Omega)|$.

The resulting mixed hybrid scheme reads as follows.
\emph{Find $\bM_h\in \XdDivnnc(\mesh)$, $u_h\in P^1(\mesh)$,
and $\bpsi_h\in \HCT_0^{2}(\cS)$ such that}
\begin{subequations} \label{KL_d1_h}
\begin{alignat}{3}
   &\vdual{\cCinv\bM_h}{\dbM} - \vdual{u_h}{\div\Div\dbM}_\mesh - \dual{\bpsi_h}{\dbM}_\cS
   &&= 0, \\
   &-\vdual{\div\Div\bM_h}{\du}_\mesh - \dual{\dbpsi}{\bM_h}_\cS
   &&= -\vdual{f}{\du}
\end{alignat}
\end{subequations}
\emph{holds for any $\dbM\in \XdDivnnc(\mesh)$, $\du\in P^1(\mesh)$,
and $\dbpsi\in \HCT_0^{2}(\cS)$.}

It converges quasi-optimally.

\begin{theorem} \label{thm_KL_d1_h}
Let $f\in L_2(\Omega)$ be given. System \eqref{KL_d1_h} is well posed.
Its solution $(\bM_h,u_h,\bpsi_h)$ satisfies,
\[
   \|\bM-\bM_h\|_{\dDiv,\mesh} + \|u-u_h\| + \|\bpsi-\bpsi_h\|_{3/2,1/2,\cS}
   \le C \Bigl(\|\bM-\bQ\|_{\dDiv,\mesh} + \|u-v\| + \|\bpsi-\bphi\|_{3/2,1/2,\cS}\Bigr)
\]
for any $\bQ\in \XdDivnnc(\mesh)$, $v\in P^1(\mesh)$, and $\bphi\in \HCT_0^{2}(\cS)$.
Here, $(\bM,u,\bpsi)$ is the solution of \eqref{KL_d1} and $C>0$ is independent of
$\mesh$, $\bQ$, $v$, and $\bphi$.
\end{theorem}

A proof of this theorem is given in abstract form by Theorem~\ref{thm_gmh}.
By \cite[Proposition 4]{FuehrerH_MKL}, relation \eqref{subh} holds, so that we
only have to check the existence of Fortin operator components $\opF_1$, $\opF_2$
that satisfy \eqref{Fgm}. This is done in the remainder of this section,
see \eqref{opFd1} and Lemma~\ref{la_F_d1} below.

Element $\XdDiv(\el)$ has the following $15$ degrees of freedom, cf.~\eqref{trdd},
\begin{subequations} \label{dof_KL}
\begin{alignat}{3}
   &\dual{\bn\cdot\Div\bM+\partial_t(\bt\cdot\bM\bn)}{\phi}_E,\quad
   && \phi\in P^1(E),\ E\in\cE, \label{dof_KL1}\\
   &\dual{\bn\cdot\bM\bn}{\phi}_E,
   && \phi\in P^1(E),\ E\in\cE, \label{dof_KL2}\\
   & [\bt\cdot\bM\bn]_{\partial\el}(x),
   && x\in\cN(\el) \label{dof_KL3}.
\end{alignat}
\end{subequations}
These degrees, taken for every element $\el\in\mesh$ with unique values
for edges $E\in\cE$ and vertices $x\in\cN$, subject to constraints \eqref{trdd4}
at interior vertices, define an interpolation operator
\begin{align} \label{Idd_map}
   \Idd_\mesh:\;
   H(\dDiv,\Omega;\SS)\cap H^r(\Omega;\SS)\to \XdDiv(\mesh)\cap H(\dDiv,\Omega;\SS)
\end{align}
for $r>3/2$. According to \cite[Proposition~10]{FuehrerH_MKL},
this interpolation operator satisfies
\begin{align}
   &\div\Div\Idd_\mesh = \Pi^1_\mesh\div\Div, \label{Idd_comm}\\
   &\|\Idd_\mesh\bM\| \lesssim \|\bM\|_{r}\quad
   \forall\bM\in  H(\dDiv,\Omega;\SS)\cap H^r(\Omega;\SS)
   \quad(r>3/2). \label{Idd_bound}
\end{align}
We use the reduced element $\XdDivnnc(\el)$ with only constant moments in \eqref{dof_KL2}.
Inspection of the details in \cite{FuehrerH_MKL} reveals that the corresponding
interpolation operator $\Iddnnc_\mesh$ satisfies
properties \eqref{Idd_comm} and \eqref{Idd_bound} as well. In particular, the commutativity
property can be seen by evaluating $\vdual{\div\Div\Iddnnc_\mesh\bM}{v}_\el$
for $v\in P^1(\el)$ and integrating by parts. The normal-normal trace of $\bM$
meets the normal derivative of $v$, an edge-wise constant, and thus gives only rise to the
constant moments in \eqref{dof_KL2}. The fact that $\Iddnnc_\mesh$
maps $H(\dDiv,\Omega;\SS)\cap H^r(\Omega;\SS)$ (for $r>1/2$) to
$\XdDivnnc(\mesh)\cap H(\dDiv,\Omega;\SS)$ holds by \eqref{Idd_map} and the selection
of the degrees of freedom.

Element $\HCT(\el)$ has dimension $9$. The canonical degrees of freedom of $v\in\HCT(\el)$ are
\begin{equation} \label{dof_HCT_can}
   v(x), \grad v(x)\quad (x\in\cN(\el))\quad\text{for}\ \el\in\mesh,
\end{equation}
cf.~\cite{Ciarlet_78_IEE}. They uniquely define a piecewise cubic polynomial
trace $v|_{\partial\el}$ and a piecewise linear normal derivative
$\partial_n v|_{\partial\el}$. In our numerical experiments we do this
calculation ``on the fly''.

For our analysis of Fortin operators we continue to identify different degrees
that are dual to the degrees \eqref{dof_KL1}, \eqref{dof_KL3} of $\XdDivnnc(\el)$.

\begin{lemma} \label{la_dof_HCT}
A function $v\in\HCT(\el)$ with $\el\in\mesh$ is uniquely defined by
\[
   v(x)\quad (x\in\cN(\el))\quad\text{and}\quad
   \dual{\dpsi}{v}_E\quad (\dpsi\in P^1(E),\ E\in\cE(\el)).
\]
\end{lemma}

\begin{proof}
Since the dimension of $\HCT(\el)$ is equal to the number of specified degrees of freedom,
it is enough to show that an element $v\in\HCT(\el)$ with vanishing degrees is zero.
For given $\el\in\mesh$ with vertices
$x_j$ and opposite edges $E_j$, $j=1,2,3$, let $\lambda_j$, $j=1,2,3$,
be the barycentric coordinates.
By definition, $v|_E\in P^3(E)$ for any $E\in\cE(\el)$.
Setting $v(x_j)=0$, $j=1,2,3$,
means that $v|_{E_j}$ is a cubic polynomial that vanishes at the endpoints
$x_{j+1}$, $x_{j+2}$, $j=1,2,3$ (we use a numbering modulo $3$).
Therefore, $v|_{E_j}\in
\bigl(\lambda_{j+1}\lambda_{j+2}\, \mathrm{span}\{1,\lambda_{j+2}-\lambda_{j+1}\}\bigr)|_{E_j}$,
$j=1,2,3$. Orthogonalities $\dual{1}{v}_{E_j}=\dual{\lambda_{j+2}-\lambda_{j+1}}{v}_{E_j}$
imply $v|_{E_j}=0$, $j=1,2,3$. We conclude that the canonical degrees of freedom
\eqref{dof_HCT_can} of $v$ vanish, that is, $v=0$.
\end{proof}

\begin{remark} \label{rem_dof_HCT}
The degrees of freedom given by Lemma~\ref{la_dof_HCT} are intrinsic to $\partial\el$
and therefore uniquely define functions of $\HCT(\el)$ by their traces. However,
these degrees do not imply conformity for $\HCT_0^2(\cS)$ as traces of $H^2(\Omega)$
without the use of basis functions that are conforming with the canonical degrees of
freedom \eqref{dof_HCT_can}. 
\end{remark}

We are in a position to define and analyze a Fortin operator.
Let $\bI\in\SS$ denote the identity tensor.
For given $\bM\in H(\dDiv,\mesh;\SS)$ we define
\begin{subequations} \label{opFd1}
\begin{alignat}{2}
   &\opF_1(\bM):=\Iddnnc_\mesh(z\bI)\quad
   &&\text{with}\quad
   z\in H^1_0(\Omega):\ \Delta z=\Pi^1_\mesh\div_\mesh\Div_\mesh\bM, \label{opFd1a}\\
   &\opF_2(\bM):=\bQ:=(\bQ_\el)_\el
   &&\text{with}\quad
   \bQ_\el\in\XdDiv(\el):\ \bn\cdot\bQ_\el\bn|_{\partial\el}=0, \label{opFd1b}\\
   &&& \dual{\trtwoloc(v)}{\bQ_\el}_{\partial\el} = \dual{\trtwoloc(v)}{\bM|_\el}_{\partial\el}
   \ \forall v\in \HCT(\el),\ \el\in\mesh. \nonumber
\end{alignat}
\end{subequations}

\begin{lemma} \label{la_F_d1}
Operators $\opF_1$, $\opF_2$ satisfy conditions \eqref{Fgm}.
\end{lemma}

\begin{proof}
The construction of operator $\opF_1$ stems from \cite[Proof of Theorem~12]{FuehrerH_MKL}.
It is well defined by the properties of $\Iddnnc_\mesh$ because $z\in H^{r}(\Omega)$ with $r>3/2$.
Relation \eqref{Fgm1} follows by using commutativity property \eqref{Idd_comm} for $\Iddnnc_\mesh$
and noting that $\div_\mesh\Div_\mesh(z\bI)=\Delta z$.
Boundedness \eqref{Fgm3} follows by the same relation, together with boundedness
\eqref{Idd_bound} for $\Iddnnc_\mesh$ and stability $\|z\|_{r}\lesssim \|\div\Div\bM\|_\mesh$.

Operator $\opF_2$ is well defined (on every element $\el\in\mesh$).
In fact, the defining right-hand side is a duality pairing on $H^2(\el)\times H(\dDiv,\el;\SS)$,
and the degrees of freedom \eqref{dof_KL1}, \eqref{dof_KL3} of $\XdDivnnc(\el)$
(setting degrees \eqref{dof_KL2} to zero) and those of $HCT(\el)$ are dual to each other,
see representation \eqref{IP_d1} and Lemma~\ref{la_dof_HCT}.
By definition, $\bQ\in H(\dDiv,\mesh;\SS)$.
Given any $v\in \HCT(\mesh)\cap H^2(\Omega)$, $\bQ$ satisfies
\begin{equation*}
   \dual{\trtwo(v)}{\bQ}_\cS = \sum_{\el\in\mesh}\dual{\trtwoloc(v)}{\bQ_\el}_{\partial\el}
   = \sum_{\el\in\mesh}\dual{\trtwoloc(v)}{\bM|_\el}_{\partial\el}
   = \dual{\trtwo(v)}{\bM}_\cS.
\end{equation*}
This proves \eqref{Fgm2}.
It remains to check the boundedness of $\opF_2$.
The bound $\|\bQ\| \lesssim \|\bM\|_{\dDiv,\mesh}$ holds by a finite-dimension argument
on every element and scaling properties, cf.~\cite[Proof of Lemma~16]{FuehrerH_19_FDD}.
The bound $\|\div\Div\bQ\|_\mesh\le \|\div\Div\bM\|_\mesh$ is due to commutativity
property \eqref{Idd_comm}.
\end{proof}

\begin{remark}
We have proved the discrete inf-sup condition \eqref{infsup_h} by constructing
Fortin operators satisfying \eqref{Fgm}. Brezzi and Marini used an
$H(\dDiv,\mesh;\SS)$-extension of right-hand side function $f$, thus avoiding
the field variable $u$. In their case, only an inf-sup condition for
bilinear form $\dual{\dbpsi}{\bM}_\cS$ is required.
Theorem~3.7 from \cite{BrezziM_75_NSP} with
$m=1$, $r=3$, $s=1$ (the respective polynomial degrees of tensors on $\el$,
and the trace and normal derivative variables on edges $E\in\cE(\el)$)
proves the discrete inf-sup condition for the pair
$\XdDiv(\el)\times\trtwo(\HCT(\el))$
by noting that $P^1(\el;\SS)\subset\XdDiv(\el)$, cf.~\cite[Proposition~4]{FuehrerH_MKL}.
Our reduced element $\XdDivnnc(\el)$ does not comprise $P^1(\el;\SS)$ and is thus not covered
by \cite[Theorem~3.7]{BrezziM_75_NSP}.
\end{remark}

\subsection{Normal-normal continuous mixed formulation} \label{sec_KL_d2}

We present a formulation that provides the variational framework for a type of
Hellan--Herrmann--Johnson method that controls bending moments in the product energy
space $H(\dDiv,\mesh;\SS)$ rather than in $L_2(\Omega;\SS)$
augmented with traces, cf.~\cite{Hellan_67_AEP,Herrmann_67_FEB,Johnson_73_CMF}.
For a variational formulation in Banach spaces see \cite[Cap.~10.3]{BoffiBF_13_MFE}.

Recall trace operator $\trtwo$ introduced in \eqref{trtwo}. We select
\begin{align*}
   \sH(A^*,\mesh)
   :=\Hnn(\dDiv,\mesh;\SS)
   &:=\{\bQ\in H(\dDiv,\mesh;\SS);\; \dual{\trtwo(v)}{\bQ}_\cS=0\
         \forall v\in H^2_0(\Omega)\cap H_0^1(\mesh)\}\\
   &\subset H(A^*,\mesh)=H(\dDiv,\mesh;\SS)
\end{align*}
with $H^1_0(\mesh):=\Pi_{\el\in\mesh} H^1_0(\el)$ and denote
\begin{align*}
   \trtone:=\trAS,\quad
   H^{3/2}_0(\cS) := H_0(A,\cS)=\trtone\bigl(H^2_0(\Omega)\bigr),\quad
   \|\cdot\|_{3/2,\cS}:=\|\cdot\|_{A,\cS}.
\end{align*}
Notation $\Hnn(\dDiv,\mesh;\SS)$ indicates that sufficiently smooth elements have
continuous normal-normal traces across $\cS$ whereas notation $\trtone$ refers to
the fact that it is the canonical trace operator from $H^1(\Omega)$ onto $\cS$,
restricted to $H^2(\Omega)$ and with stronger norm. For a non-trivial mesh $\mesh$,
$\Hnn(\dDiv,\mesh;\SS)$ is a strict subspace of $H(\dDiv,\mesh;\SS)$.

The generalized mixed hybrid formulation \eqref{gm} becomes the
\emph{normal-normal continuous mixed formulation} of the Kirchhoff--Love model:
\emph{Find $\bM\in \Hnn(\dDiv,\mesh;\SS)$, $u\in L_2(\Omega)$,
and $\psi\in H_0^{3/2}(\cS)$ such that}
\begin{subequations} \label{KL_d2}
\begin{alignat}{3}
   &\vdual{\cCinv\bM}{\dbM} - \vdual{u}{\div\Div\dbM}_\mesh - \dual{\psi}{\dbM}_\cS
   &&= 0, \label{KL_d2a}\\
   &-\vdual{\div\Div\bM}{\du}_\mesh - \dual{\dpsi}{\bM}_\cS
   &&= -\vdual{f}{\du} \label{KL_d2b}
\end{alignat}
\end{subequations}
\emph{holds for any $\dbM\in \Hnn(\dDiv,\mesh;\SS)$, $\du\in L_2(\Omega)$,
and $\dpsi\in H_0^{3/2}(\cS)$.}

\begin{theorem} \label{thm_KL_d2}
Let $f\in L_2(\Omega)$ be given.
Problem \eqref{KL_d2} is well posed. Its solution $(\bM,u,\psi)$ satisfies
\[
   \|\bM\|_{\dDiv,\mesh} + \|u\| + \|\psi\|_{3/2,\cS} \le C \|f\|
\]
with a constant $C$ that is independent of $f$ and $\mesh$.
Furthermore, $u\in H^2_0(\Omega)$, $\bM=\cC D^2 u\in H(\dDiv,\Omega;\SS)$,
$\psi=\trtone(u)$, and $u$, $\bM$ solve \eqref{KLove}.
\end{theorem}

\begin{proof}
The proof is identical to the proof of Theorem~\ref{thm_KL_d1}.
\end{proof}

\subsubsection*{Discretization}

Relation \eqref{IP_d1} shows that $u\in H^2_0(\Omega)$ and a sufficiently smooth tensor
$\bM\in \Hnn(\dDiv,\mesh;\SS)$ satisfy
\begin{equation} \label{IP_d2}
   \dual{\trtone(u)}{\bM}_\cS
   = \sum_{\el\in\mesh}
          \Bigl(\sum_{x\in\cN(\el)} [\bt\cdot\bM\bn]_{\partial\el}(x) u(x)
              - \dual{\bn\cdot\Div\bM+\partial_t(\bt\cdot\bM\bn)}{u}_{\partial\el}\Bigr).
\end{equation}
As already indicated, we conclude that trace $\psi=\trtone(u)$ reduces to the canonical trace
of $u$ onto $\cS$, measured in a stronger norm than the canonical trace operator acting on
$H^1(\Omega)$.
We approximate $\psi$ by traces of the reduced HCT-element \eqref{HCT},
\[
   \HCT_0^{2,1}(\cS) := \trtone\Bigl(\HCT(\mesh)\cap H^2_0(\Omega)\Bigr),
\]
and use the reduced-element space $\XdDivnnc(\mesh)$ with continuous normal-normal
traces to approximate $\bM$,
\begin{align*}
   \XdDivnn(\mesh) &:= \XdDivnnc(\mesh)\cap \Hnn(\dDiv,\mesh;\SS).
\end{align*}
In our abstract notation, the approximation spaces are
\[
   \sH_h(A^*,\mesh):= \XdDivnn(\mesh),\quad
   H_h(\mesh) := P^1(\mesh),\quad
   H_h(A,\cS) := \HCT_0^{2,1}(\cS).
\]
They have dimensions $|\cE|+9|\mesh|$, $3|\mesh|$ and $3|\cV(\Omega)|$, respectively.

The resulting normal-normal continuous mixed scheme reads as follows.
\emph{Find $\bM_h\in \XdDivnn(\mesh)$, $u_h\in P^1(\mesh)$,
and $\psi_h\in \HCT_0^{2,1}(\cS)$ such that}
\begin{subequations} \label{KL_d2_h}
\begin{alignat}{3}
   &\vdual{\cCinv\bM_h}{\dbM} - \vdual{u_h}{\div\Div\dbM}_\mesh - \dual{\psi_h}{\dbM}_\cS
   &&= 0, \\
   &-\vdual{\div\Div\bM_h}{\du}_\mesh - \dual{\dpsi}{\bM_h}_\cS
   &&= -\vdual{f}{\du}
\end{alignat}
\end{subequations}
\emph{holds for any $\dbM\in \XdDivnn(\mesh)$, $\du\in P^1(\mesh)$,
and $\dpsi\in \HCT_0^{2,1}(\cS)$.}

It converges quasi-optimally.

\begin{theorem} \label{thm_KL_d2_h}
Let $f\in L_2(\Omega)$ be given. System \eqref{KL_d2_h} is well posed.
Its solution $(\bM_h,u_h,\psi_h)$ satisfies,
\[
   \|\bM-\bM_h\|_{\dDiv,\mesh} + \|u-u_h\| + \|\psi-\psi_h\|_{3/2,\cS}
   \le C \Bigl(\|\bM-\bQ\|_{\dDiv,\mesh} + \|u-v\| + \|\psi-\phi\|_{3/2,\cS}\Bigr)
\]
for any $\bQ\in \XdDivnn(\mesh)$, $v\in P^1(\mesh)$, and $\phi\in \HCT_0^{2,1}(\cS)$.
Here, $(\bM,u,\psi)$ is the solution of \eqref{KL_d2} and $C>0$ is independent of
$\mesh$, $\bQ$, $v$, and $\phi$.
\end{theorem}

The proof of Theorem~\ref{thm_KL_d1_h} applies in this case as well.
In fact, Fortin operator component $\opF_1$ from \eqref{opFd1a} maps to
$H(\dDiv,\Omega;\SS)\subset\Hnn(\dDiv,\mesh;\SS)$ and
component $\opF_2$ from \eqref{opFd1b} maps to
$\Hnn(\dDiv,\mesh;\SS)$ since it sets degrees of freedom \eqref{dof_KL2} to zero.

\subsection{Numerical experiments} \label{sec_KL_num}

We consider a simple example of problem \eqref{KLove}
with domain $\Omega=(0,1)^2$ and polynomial solution
$u(x_1,x_2)=x_1^2(1-x_1)^2x_2^2(1-x_2)^2$. Using the identity tensor for $\cC$,
the bending moments and right-hand side function
are $\bM=D^2 u$ and $f=\div\Div\bM=\Delta^2 u$, respectively.
We use the discretization spaces as specified in the respective sections,
with uniform meshes of size $h:=N^{-1/2}$ where $N:=|\mesh|$.
We present results for the nodal-continuous and continuous primal hybrid methods
\eqref{KL_p2_h}, \eqref{KL_p3_h}, and the mixed hybrid and normal-normal continuous
mixed methods \eqref{KL_d1_h}, \eqref{KL_d2_h}.
In all the cases,
the domain bilinear forms are calculated analytically on a reference element
and Piola--Kirchhoff transformation (with appropriate scalings) onto elements,
cf.~\cite{FuehrerHN_19_UFK,CarstensenH_NNC}.
We use numerical integration for the right-hand side entries $\vdual{f}{\du}_\el$ (7-point Gauss)
and element error calculation (16-point Gauss), cf.~\cite{Dunavant_85_HDS}.
We approximate the skeleton bilinear forms by the 5-point Gauss formula on every edge,
and use central differences for the derivatives
of the effective shear force (normal component of $\Div\bM$ and tangential
derivative of $\bt\cdot\bM\bn$), required for schemes \eqref{KL_d1_h} and \eqref{KL_d2_h}.

All the discretizations aim at lowest-order approximations, and this is confirmed
by our numerical results, with some superconverging components that we do not
analyze here. Figure~\ref{fig_primal_nodal_cont_err} shows the errors
$\|u-u_h\|$ (``u'') and $\|D^2(u-u_h)\|_\mesh$ (``D$^2$u'')
for the nodal-continuous primal hybrid method
\eqref{KL_p2_h} along with curves of orders $O(h)=O(N^{-1/2})$ and $O(h^2)$,
indicating $\|u-u_h\|_{2,\mesh}=O(h)$ and superconvergence $\|u-u_h\|=O(h^2)$.
For illustration we also plot weighted $L_2(\cS)$-errors for the approximations
of the traces of $\bM$ on the skeleton.
As indicated after \eqref{KL_p2_h}, we denote
by $\eta^\mathit{nn}_h$ and $\eta^\mathit{sf}_h$ the components of $\boldeta_h$
that correspond to the normal-normal trace and the effective shear force, respectively.
Curves ``Mnn'' and ``shear'' present the errors
$\|h_\cS^{1/2}(\bn\cdot\bM\bn-\eta^\mathit{nn}_h)\|_\cS$ and
$\|h_\cS^{3/2}(\bn\cdot\Div\bM+\partial_t(\bt\cdot\bM\bn)-\eta^\mathit{sf}_h)\|_\cS$.
Here, $h_\cS|_E:=|E|$ for $E\in\cE$. Both curves are of order $O(h)$. 
The weightings are chosen to have the respective scalings of the edge-wise
$H^{-1/2}$ and $H^{-3/2}$ norms, cf.~\cite{Heuer_14_OEF}. We conclude that
the results indicate convergence $\|\trdDivM(\bM)-\boldeta_h\|_{-3/2,-1/2,J0,\cS}=O(h)$.
A numerical confirmation of this result would require to construct appropriate extensions
of $\boldeta_h$ to elements of $H(\dDiv,\Omega;\SS)$
and the calculation of the error in this norm, cf.~\eqref{tracenorm_p2}.
For the relevance in applications of the traces of $\bM$, in particular of the effective
shear force, we present their approximations on the left in
Figures~\ref{fig_primal_nodal_cont_Mnn} (approximation of $\bn\cdot\bM\bn$)
and~\ref{fig_primal_nodal_cont_sfM} (absolute value of approximation of effective shear force).
In both cases, the right figures show (with the same scale as on the respective left side)
their difference with the piecewise-constant $L_2$-projections of the exact values
(absolute difference in the latter case). 
In this example with smooth solution, we observe point-wise convergence of
the $\bn\cdot\bM\bn$ approximation, and point-wise control of the approximation of
the effective shear force, essentially an $H^{-3/2}$-functional.
In the case of the continuous primal hybrid method \eqref{KL_p3_h}, the results
are shown in Figure~\ref{fig_primal_cont_err} and are analogous.
In this case, the trace variable provides only an approximation
of $\bn\cdot\bM\bn$, with behavior as before.

Figures~\ref{fig_dual_err} and~\ref{fig_dual_nn_err} present the results
for the mixed hybrid and the normal-normal continuous mixed methods, respectively.
We show the curves for the errors $\|u-u_h\|$ (``u''), $\|\bM-\bM_h\|$ (``M''),
$\|\div\Div(\bM-\bM_h)\|_\mesh$ (``divDiv M'') and $\|D^2 u-\strain(G_h)\|$ (``D$^2$u''),
along with lines indicating $O(h)$ and $O(h^2)$. Here, $G_h$ is the $P^1(\mesh;\R^2)$
approximation of $\grad u$, given explicitly by the gradient unknowns of trace approximation
$\bpsi_h$ (in the case of scheme \eqref{KL_d1_h}) or $\psi_h$ (in the case of scheme
\eqref{KL_d2_h}). The numerical results indicate convergence order $O(h)$ of
$\|\bM-\bM_h\|$ and the trace approximations of $u$ in both cases,
and increased convergence order $O(h^2)$ of $\|u-u_h\|$ and $\|\div\Div(\bM-\bM_h)\|_\mesh$.
We have not shown the superconvergence of $u_h$ but the convergence
$\|\div\Div(\bM-\bM_h)\|_\mesh=O(h^2)$ holds by construction,
$\div_\mesh\Div_\mesh\bM_h=\Pi^1_\mesh f$.
We note that the observation $\|D^2 u-\strain(G_h)\|=O(h)$ means that
the approximation of $\partial_n u|_\cS$ in a skeleton space $H^{1/2}(\cS)$
(normal derivatives on $\cS$ of $H^2(\Omega)$-functions)
is of this order. A direct control of the approximation of $u|_\cS$ in
$H^{3/2}(\cS):=H^2(\Omega)|_\cS$ would require to provide an $H^2(\Omega)$-extension
of the corresponding component of $\bpsi_h$ or $\psi_h$. There are no simple low-order polynomial
elements to do this, which is precisely the reason to use composite HCT-elements
for domain-based approximations of $u\in H^2(\Omega)$.
For illustration, we have also implemented scheme \eqref{KL_d2_h} with the full
$\XdDiv$-element from \cite{FuehrerH_MKL} rather than reduced element $\XdDivnn$ for
the bending moment $\bM$. The curve labelled as ``M(15)'' in Figure~\ref{fig_dual_nn_err}
($15$ refers to the dimension of the full element) refers to this case and
indicates the improved convergence $\|\bM-\bM_h\|=O(h^2)$.

For comparison, Table~\ref{tab} lists the $L_2$-errors for the approximation
of the deflection in case of the four methods considered in this section,
both primal and both mixed methods. We have already seen that $\|u-u_h\|=O(N^{-2})$
in all the cases. The table shows that the continuous primal method
leads to slightly smaller deflection errors than the nodal-continuous primal method 
whereas the mixed methods show increased errors by factors of about $4.5$
and $2$ for the mixed hybrid and normal-normal mixed hybrid methods, respectively.

\begin{figure}[htb]
\begin{center}
\includegraphics[width=0.9\textwidth,height=0.6\textwidth]
{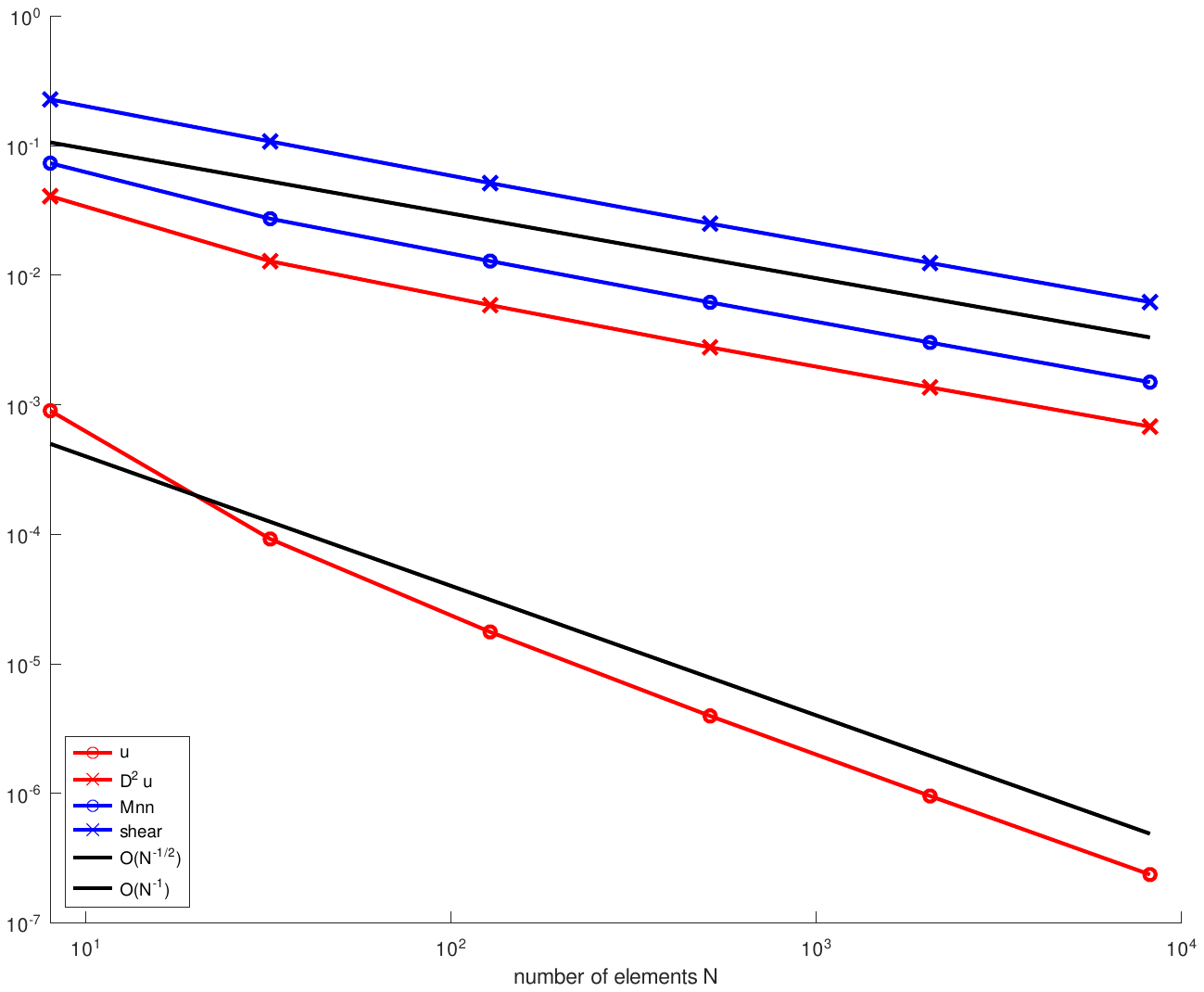}
\end{center}
\vspace{-1cm}
\caption{Errors for the nodal-continuous primal hybrid method \eqref{KL_p2_h}.
         The curves are
         ``u'': $\|u-u_h\|$, ``D$^2$u'': $\|D^2(u-u_h)\|_\mesh$,
         ``Mnn'': $L_2(\cS)$-error for normal-normal traces of $\bM$, weighted with $h^{1/2}$,
         ``shear'': $L_2(\cS)$-error for effective shear force approximation,
                    weighted with $h^{3/2}$, and curves indicating $O(h)$, $O(h^2)$.
}
\label{fig_primal_nodal_cont_err}
\end{figure}

\begin{figure}[htb]
\vspace{-1cm}
\begin{center}
\includegraphics[width=0.9\textwidth]{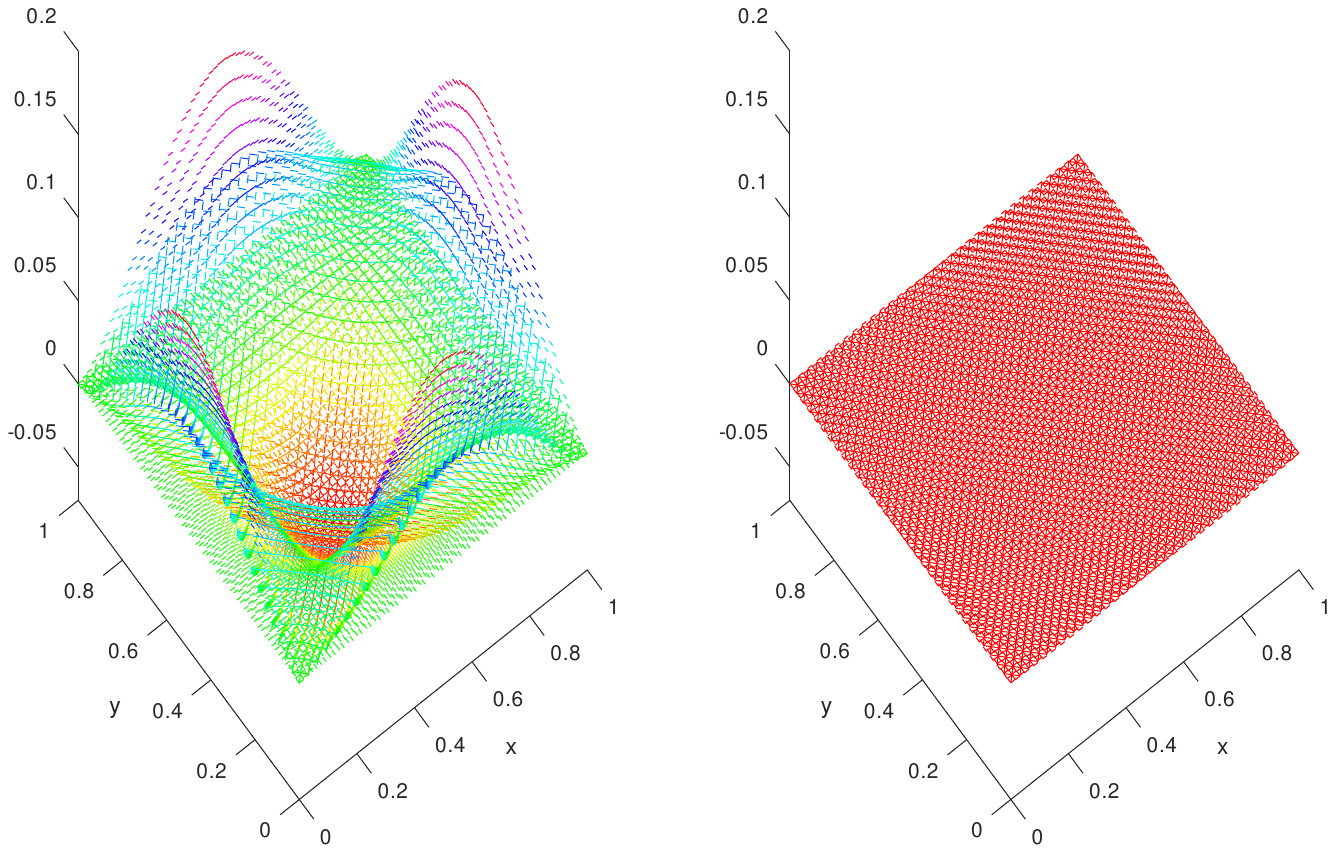}
\end{center}
\vspace{-1cm}
\caption{The approximation of trace component $\bn\cdot\bM\bn|_\cS$ from
         the nodal-continuous primal hybrid method \eqref{KL_p2_h} (on the left)
         and the difference of p/w constant $L^2$-projection of $\bn\cdot\bM\bn|_\cS$ and
         its approximation (on the right).
         The mesh has 8192 elements and 12416 edges.}
\label{fig_primal_nodal_cont_Mnn}
\end{figure}

\begin{figure}[htb]
\vspace{-1cm}
\begin{center}
\includegraphics[width=0.9\textwidth,height=0.6\textwidth]
{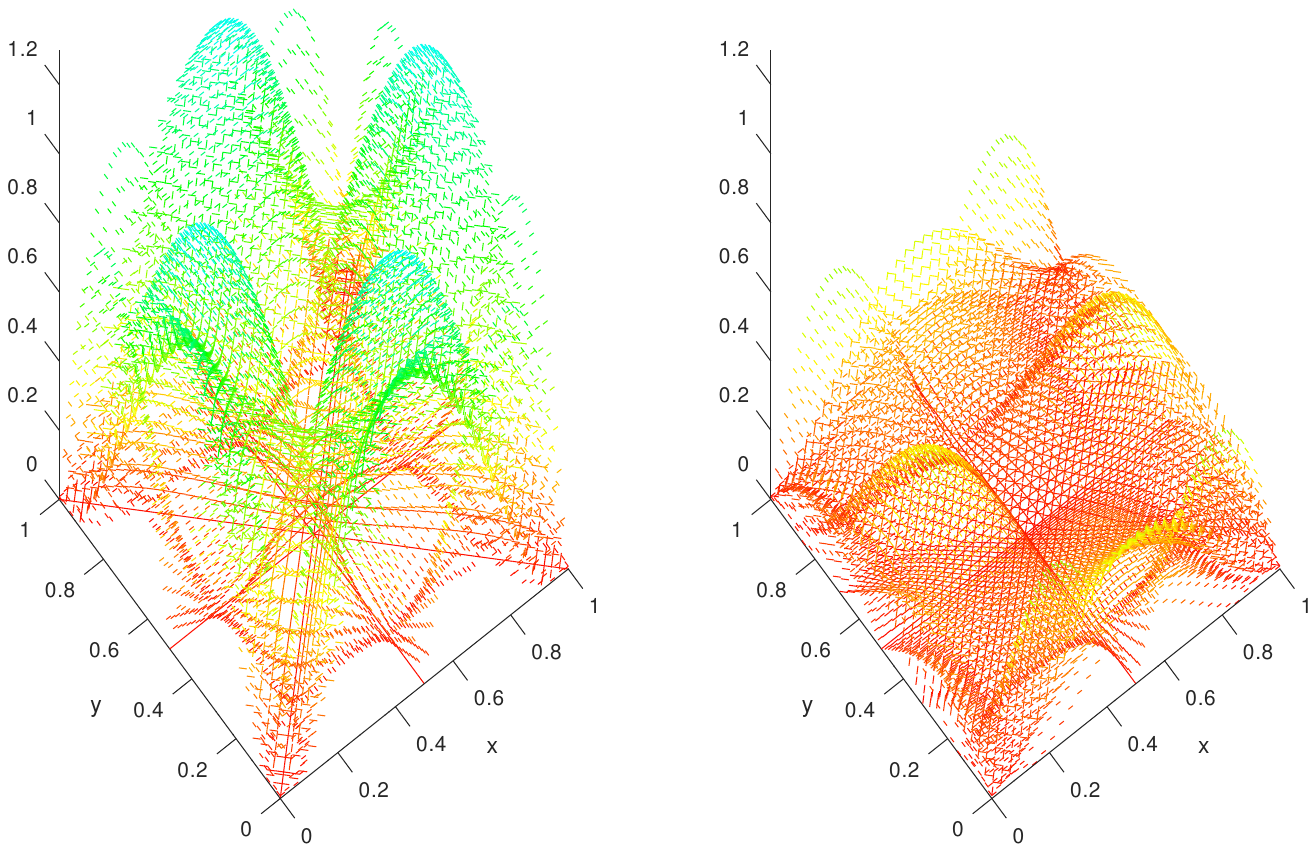}
\end{center}
\vspace{-1cm}
\caption{The approximation of the effective shear force from
         the nodal-continuous primal hybrid method \eqref{KL_p2_h} (absolute values, on the left)
         and the difference of the p/w constant $L_2$-projection of the effective shear force
         and its approximation (absolute values of the difference, on the right).
         The mesh has 8192 elements and 12416 edges.}
\label{fig_primal_nodal_cont_sfM}
\end{figure}

\begin{figure}[htb]
\vspace{-1cm}
\begin{center}
\includegraphics[width=0.9\textwidth,height=0.6\textwidth]
{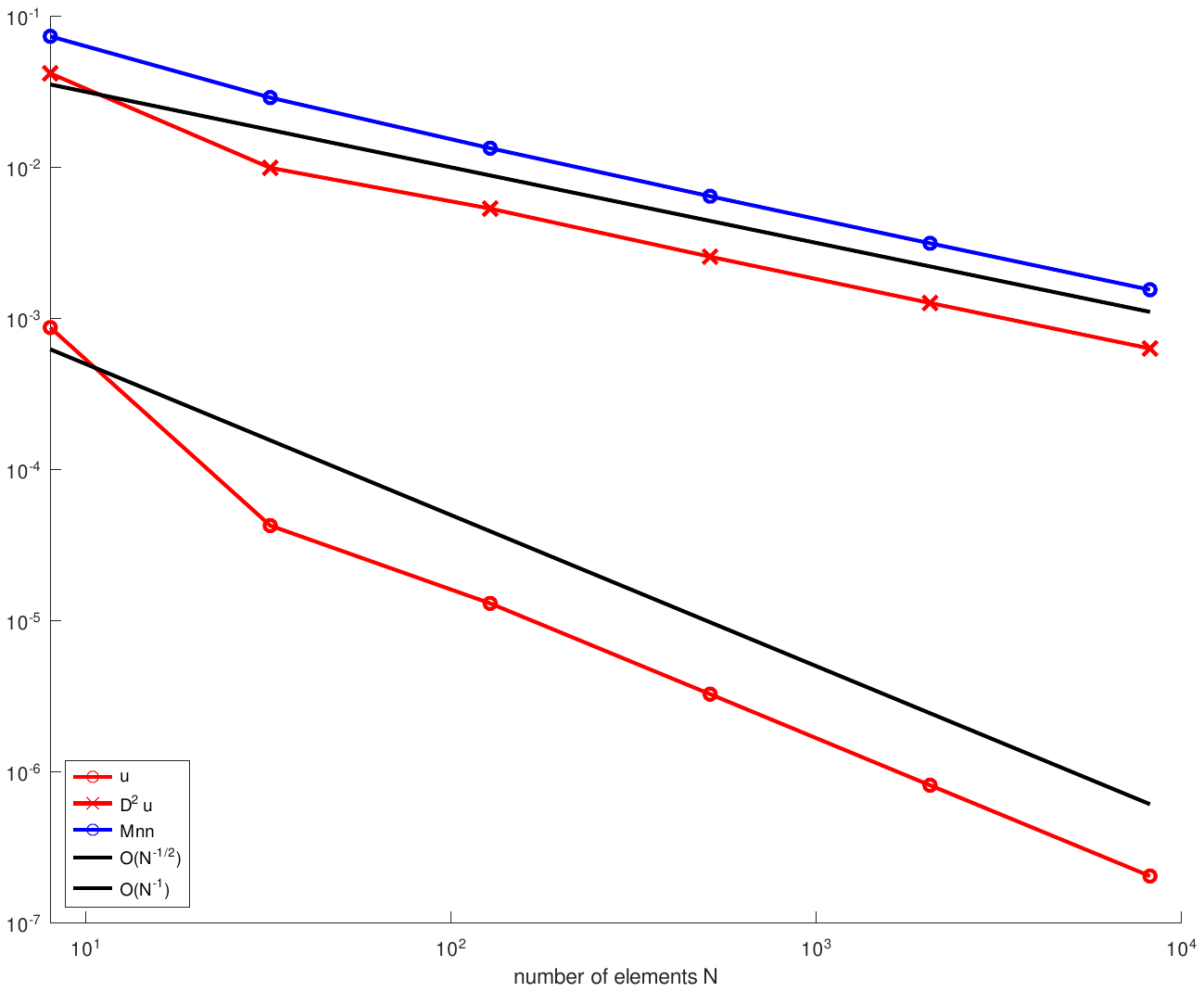}
\end{center}
\vspace{-1cm}
\caption{Errors for the continuous primal hybrid method \eqref{KL_p3_h}.
         The curves are
         ``u'': $\|u-u_h\|_{2,\mesh}$, ``Mnn'': $L_2(\cS)$-error for normal-normal traces
         of $\bM$, weighted with $h^{1/2}$, and a curve indicating $O(h)$.
}
\label{fig_primal_cont_err}
\end{figure}

\begin{figure}[htb]
\vspace{-1cm}
\begin{center}
\includegraphics[width=0.9\textwidth,height=0.6\textwidth]
{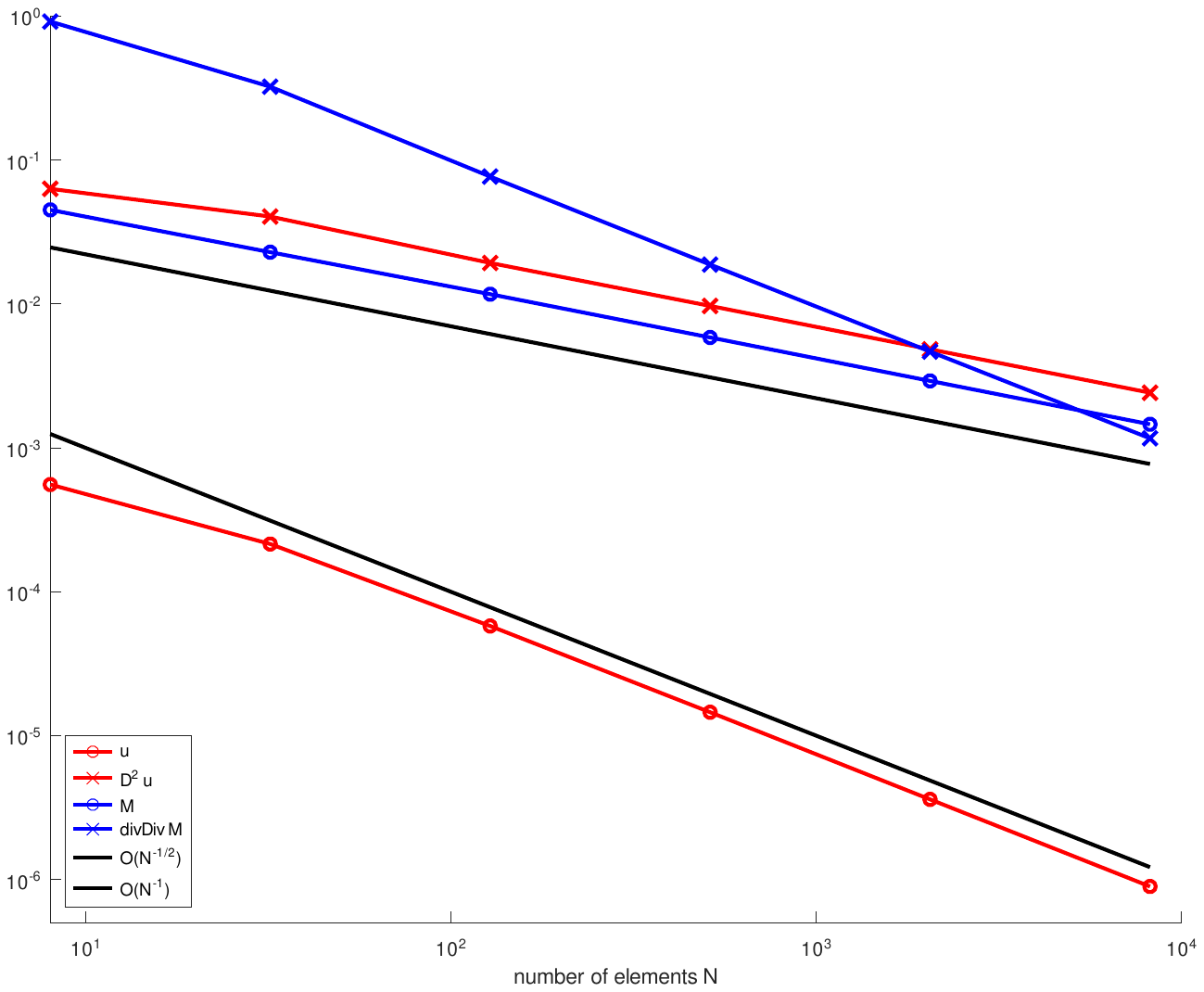}
\end{center}
\vspace{-1cm}
\caption{Errors for the mixed hybrid method \eqref{KL_d1_h}.
         The curves are
         ``u'': $\|u-u_h\|$, ``M'': $\|\bM-\bM_h\|$, ``divDiv M'': $\|f-\div\Div\bM_h\|_\mesh$,
         ``D$^2$u'': the $L_2$-error of the approximation of the Hessian induced by $\bpsi_h$,
         along with curves indicating orders $O(h)$ and $O(h^2)$.
}
\label{fig_dual_err}
\end{figure}

\begin{figure}[htb]
\vspace{-1cm}
\begin{center}
\includegraphics[width=0.9\textwidth,height=0.6\textwidth]
{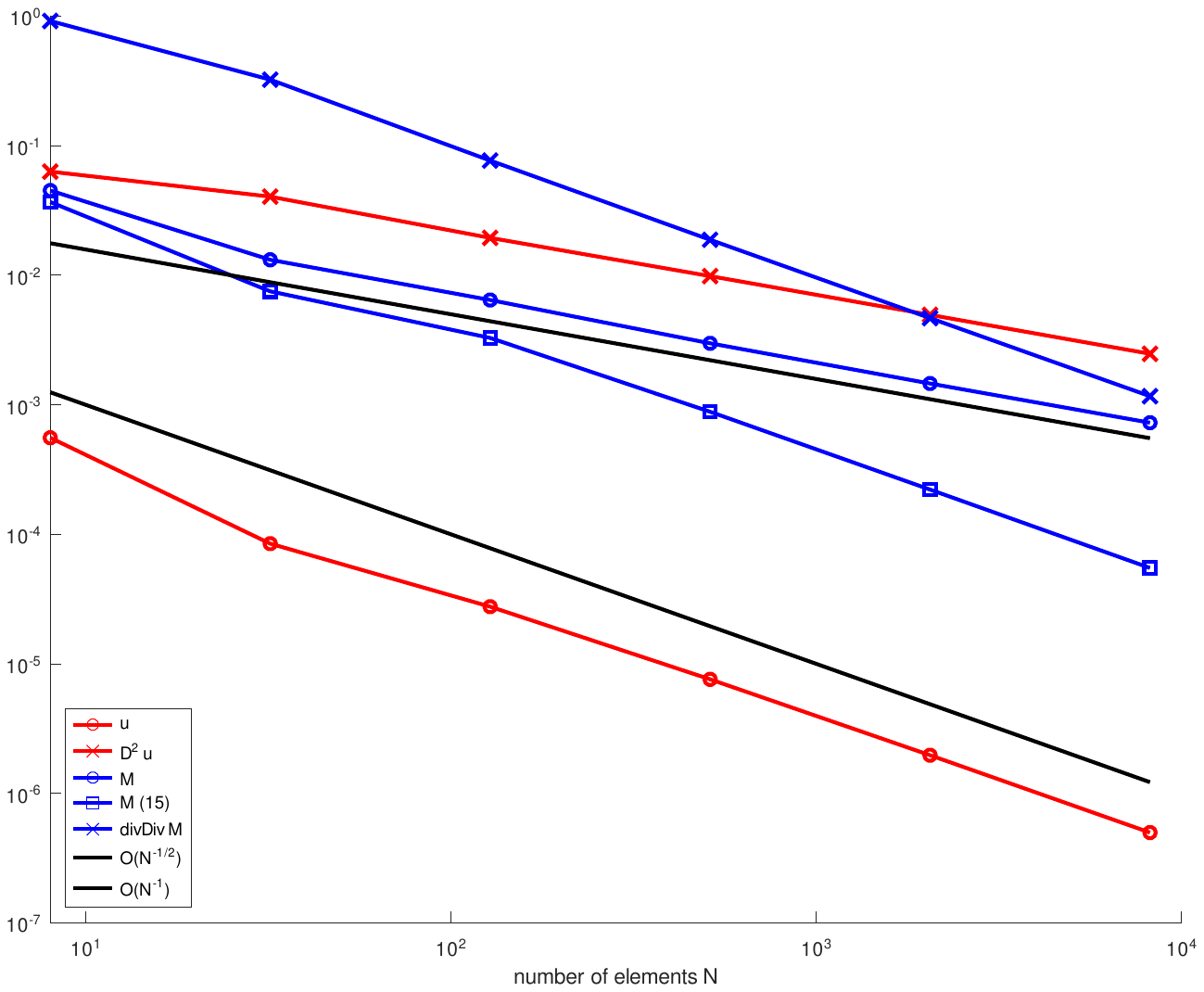}
\end{center}
\vspace{-1cm}
\caption{Errors for the normal-normal continuous mixed method \eqref{KL_d2_h}.
         The curves are
         ``u'': $\|u-u_h\|$, ``M'': $\|\bM-\bM_h\|$,
         ``M(15)'': $L_2$-error of $\bM$ with $\XdDiv(\mesh)$-approximation,
         ``divDiv M'': $\|f-\div\Div\bM_h\|_\mesh$,
         ``D$^2$u'': the $L_2$-error of the approximation of the Hessian induced by $\psi_h$,
         along with curves indicating orders $O(h)$ and $O(h^2)$.
}
\label{fig_dual_nn_err}
\end{figure}

\begin{table}
\begin{center}
\begin{tabular}{r|cccc}
 $N$ & primal(nod) & primal(cont) & mixed(hyb) & mixed(nn) \\\hline
      8 &  0.900e-03 & 0.871e-03 & 0.557e-03 & 0.640e-03\\
     32 &  0.921e-04 & 0.426e-04 & 0.215e-03 & 0.897e-04\\
    128 &  0.176e-04 & 0.130e-04 & 0.579e-04 & 0.269e-04\\
    512 &  0.396e-05 & 0.326e-05 & 0.146e-04 & 0.675e-05\\
   2048 &  0.955e-06 & 0.815e-06 & 0.361e-05 & 0.169e-05\\
   8192 &  0.236e-06 & 0.204e-06 & 0.896e-06 & 0.424e-06
\end{tabular}
\caption{
The $L_2$-errors $\|u-u_h\|$ from different methods:
nodal-continuous primal hybrid ``primal(nod)'' \eqref{KL_p2_h},
continuous primal hybrid ``primal(cont)'' \eqref{KL_p3_h},
mixed hybrid ``mixed(hyb)'' \eqref{KL_d1_h}, and
normal-normal continuous mixed hybrid ``mixed(nn)'' \eqref{KL_d2_h}.}
\label{tab}
\end{center}
\end{table}

\clearpage
\section{Proofs of abstract results} \label{sec_abstract_proofs}

We start with some preliminary results before proving
Theorems~\ref{thm_gp},~\ref{thm_gm}, Proposition~\ref{prop_split_h}, and Theorem~\ref{thm_gu}
at the end of this section.

The next statement is known for special cases, see, e.g.,~\cite[(IV.1.43)]{BrezziF_91_MHF},
\cite[(10.2.22)]{BoffiBF_13_MFE},
\cite{CarstensenDG_16_BSF}, \cite[Propositions~3.5,~3.9]{FuehrerHN_19_UFK},
\cite[Lemma~4]{FuehrerHS_20_UFR},
and in particular \cite[Lemma~A.10]{DemkowiczGNS_17_SDM} for an abstract version.

\begin{lemma} \label{la_tr}
Any $\phi\in H(A,\cS)$ and $\psi\in H(A^*,\cS)$ satisfy
$\|\phi\|_{A,\cS}=\|\phi\|_{(A^*,\sim,\mesh)^*}$ and
$\|\psi\|_{A^*,\cS}=\|\psi\|_{(A,\sim,\mesh)^*}$.
In particular, $\trAS$ and $\trAsS$ are bounded below (with constant $1$)
and the trace spaces $H(A,\cS)$ and $H(A^*,\cS)$ are closed.
\end{lemma}

\begin{proof}
We only show the norm relation for $\phi\in H(A,\cS)$. The proof for $\psi\in H(A^*,\cS)$
is analogous.

Relation $\|\phi\|_{(A^*,\sim,\mesh)^*}\le \|\phi\|_{A,\cS}$ is due to the Cauchy--Schwarz
inequality. In fact, considering $v\in H(A)$ with $\trAS(v)=\phi$ and bounding
\[
   \dual{\phi}{w}_\cS = \vdual{Av}{w}-\vdual{v}{A^*w}_\mesh
   \le \|v\|_A \|w\|_{A^*,\mesh}\quad \forall w\in H(A^*,\mesh),
\]
we find that
\[
   \|\phi\|_{(A^*,\sim,\mesh)^*} =
   \sup_{w\in \sH(A^*,\mesh),\; \|w\|_{A^*,\mesh}=1} \dual{\phi}{w}_\cS \le \|v\|_A.
\]
Taking the infimum with respect to $v\in H(A)$ subject to $\trAS(v)=\phi$ gives the result.

Let $\phi\in H(A,\cS)$ given.
It remains to show the inequality $\|\phi\|_{A,\cS}\le \|\phi\|_{(A^*,\sim,\mesh)^*}$.
To this end we first define $w\in \sH(A^*,\mesh)$ by
\begin{align} \label{pf_tr_def1}
   \vdual{A^*w}{A^*\dw}_\mesh + \vdual{w}{\dw}
   = -\dual{\phi}{\dw}_\cS\quad\forall\dw\in \sH(A^*,\mesh)
\end{align}
and then $v\in H(A)$ by
\begin{align} \label{pf_tr_def2}
   \vdual{Av}{A\dv} + \vdual{v}{\dv} = -\dual{\trAS(\dv)}{w}_\cS\quad\forall\dv\in H(A).
\end{align}
We establish some relations between $v$, $w$, and $\phi$.
\begin{enumerate}
\item Function $v$ satisfies $v=A_\mesh^*w$. To show this, let $\tilde v:=A_\mesh^*w$.
Relation~\eqref{pf_tr_def1} means that $AA_\mesh^*w=-w$ in the distributional sense
so that $\tilde v\in H(A)$. Then,
\[
   \vdual{A\tilde v}{A\dv} + \vdual{\tilde v}{\dv} =
   -\vdual{w}{A\dv} + \vdual{A^*w}{\dv}_\mesh = - \dual{\trAS(\dv)}{w}_\cS
   \quad\forall\dv\in H(A),
\]
that is, $\tilde v=v$ is the solution to \eqref{pf_tr_def2}.
\item Function $v$ has trace $\trAS(v)=\phi$. This follows with the previously seen relations
and \eqref{pf_tr_def1}, calculating
\[
   \dual{\trAS(v)}{\dw}_\cS = \vdual{Av}{\dw} - \vdual{v}{A^*\dw}_\mesh
   = -\vdual{w}{\dw} - \vdual{A^*w}{A^*\dw}_\mesh
   = \dual{\phi}{\dw}_\cS
\]
for any $\dw\in H(A^*,\mesh)$.
\item Function $v$ has norm $\|v\|_A = \|\phi\|_{A,\cS}$.
      This is due to the definition of $v$, being the minimum energy extension of its trace.
\end{enumerate}

We conclude the proof by setting $\dw:=w$ in \eqref{pf_tr_def1} and $\dv:=v$
in \eqref{pf_tr_def2} to find that
\[
   \|w\|_{A^*,\mesh}^2 = -\dual{\phi}{w}_\cS = \|v\|_A^2
\]
so that
\[
   \|\phi\|_{A,\cS} = \|v\|_A =  -\frac {\dual{\phi}{w}_\cS}{\|w\|_{A^*,\mesh}}
   \le \|\phi\|_{(A^*,\sim,\mesh)^*}.
\]
\end{proof}

For a variant of the next statement we refer to \cite[Lemma~A.9]{DemkowiczGNS_17_SDM}.
Specific versions have been considered, e.g.,
in~\cite[Propositions~III.1.1,~III.1.2]{BrezziF_91_MHF},
\cite[Propositions~3.4(i),~3.8(i)]{FuehrerHN_19_UFK}, and \cite[Proposition~5]{FuehrerHS_20_UFR}.

\begin{lemma} \label{la_reg}
Any $v\in\sH(A,\mesh)$ and $w\in\sH(A^*,\mesh)$ satisfy
\begin{align*}
   v\in H_0(A) &\quad\Leftrightarrow\quad
   \dual{\psi}{v}_\cS=0\quad\forall\psi\in H(A^*,\cS),\\
   w\in H(A^*) &\quad\Leftrightarrow\quad
   \dual{\phi}{w}_\cS=0\quad\forall\phi\in H_0(A,\cS).
\end{align*}
\end{lemma}

\begin{proof}
We show the statement for $v\in H_0(A)$. The other case is analogous.
Any $\psi\in H(A^*,\cS)$ can be written as $\psi=\trAsS(w)$ for a $w\in H(A^*)$.
By definition of $\trAsS$, $\trA$, and $H_0(A)$ we find for any $v\in H_0(A)$ that
\[
   \dual{\psi}{v}_\cS = \vdual{A^* w}{v} - \vdual{w}{Av}_\mesh
   = \vdual{A^* w}{v} - \vdual{w}{Av} = -\dual{\trA(v)}{w}_\Gamma = 0
   \quad\forall w\in H(A^*).
\]
This shows the direction ``$\Rightarrow$''.
To see the other direction let $v\in\sH(A,\mesh)$ be given with
$\dual{\psi}{v}_\cS=0$ for any $\psi\in H(A^*,\cS)$.
We calculate $Av$ in the distributional sense,
\[
   A v(w)= \vdual{v}{A^*w}
   = \dual{\trAsS(w)}{v}_\cS + \vdual{Av}{w}_\mesh
   = \vdual{Av}{w}_\mesh\quad\forall w\in C_0^\infty(\Omega;U),
\]
and conclude that $Av\in L_2(\Omega)$ so that $v\in H(A)$. To see that $v\in H_0(A)$
we calculate
\[
   \dual{\trA(v)}{w}_\Gamma = \vdual{Av}{w} - \vdual{v}{A^*w}
   = \vdual{Av}{w}_\mesh - \vdual{v}{A^*w} = \dual{\trAsS(w)}{v}_\cS = 0
\]
for any $w\in H(A^*)$ by assumption. This finishes the proof.
\end{proof}

The next result is \cite[Theorem~3.3]{CarstensenDG_16_BSF}. We just translate it to our
notation and verify the required assumptions from \cite{CarstensenDG_16_BSF}.
For a similar abstract result see \cite[Appendix~A]{GargPVdZC_14_ACF}.

\begin{lemma} \label{la_infsup}
Assume that \eqref{infsup} holds. The bilinear form
\[
   b:\;\left\{
   \begin{array}{cl}
      \sH(A^*,\mesh) \times \bigl(L_2(\Omega)\times H_0(A,\cS)\bigr) &\to \R,\\
      (w;v,\phi) &\mapsto \vdual{A^*w}{v}_\mesh + \dual{\phi}{w}_\cS
   \end{array}\right.
\]
satisfies the inf-sup property
\begin{equation*}
   \sup_{w\in\sH(A^*,\mesh),\; \|w\|_{A^*,\mesh}=1}
   b(w;v,\phi) \ge C \bigl(\|v\| + \|\phi\|_{A,\cS}\bigr)
   \quad\forall v\in L_2(\Omega),\ \forall\phi\in H_0(A,\cS)
\end{equation*}
with a constant $C>0$ that is independent of $v$, $\phi$, and $\mesh$.
\end{lemma}

\begin{proof}
We set
\begin{align*}
   &Y:=\sH(A^*,\mesh),\quad
   Y_0:=H(A^*),\quad
   X_0:= L_2(\Omega),\quad
   \hat X:= H_0(A,\cS),\\
   &b_0(v,w):=\vdual{A^*w}{v}_\mesh,\quad
   \hat b(\phi,w):= \dual{\phi}{w}_\cS
   \quad\text{for}\ v\in X_0,\ \phi\in \hat X,\ w\in Y.
\end{align*}
Switching to our notation, Assumptions~3.1 and~3.2 in \cite{CarstensenDG_16_BSF} read as
\begin{align} \label{A1}
   &\sup_{0\not=w\in Y_0} \frac {b_0(v,w)}{\|w\|_{Y}}
   := \sup_{0\not=w\in H(A^*)} \frac {\vdual{A^*w}{v}_\mesh}{\|w\|_{A^*}}
   \ge
   c_0 \|v\| =: c_0 \|v\|_{X_0} \quad\forall v\in X_0
\end{align}
and
\begin{subequations} \label{A2}
\begin{equation} \label{A2a}
\begin{split}
   Y_0:=H(A^*)
   &= \{w\in \sH(A^*,\mesh);\; \dual{\phi}{w}_\cS=0\quad \forall\phi\in H_0(A,\cS)\},\\
   &=: \{w\in Y;\; \hat b(\phi,w)=0\quad \forall\phi\in \hat X\},
\end{split}
\end{equation}
\begin{equation} \label{A2b}
   \sup_{0\not=w\in Y} \frac {\hat b(\phi,w)}{\|w\|_{Y}}
   :=
   \sup_{0\not=w\in\sH(A^*,\mesh)} \frac {\dual{\phi}{w}_\cS}{\|w\|_{A^*,\mesh}}
   \ge c\|\phi\|_{A,\cS} =: c\|\phi\|_{\hat X}\quad\forall\phi\in\hat X.
\end{equation}
\end{subequations}
Inf-sup property \eqref{A1} holds by assumption \eqref{infsup} with constant $c_0=\cis$
independent of $v$ (and $\mesh$),
Lemma~\ref{la_reg} proves \eqref{A2a}, and \eqref{A2b} holds by Lemma~\ref{la_tr} with $c=1$.
The statement follows by \cite[Theorem~3.3]{CarstensenDG_16_BSF}.
\end{proof}

\subsection{Proof of Theorem~\ref{thm_gp}} \label{pf_thm_gp}

Problem \eqref{gp} is a mixed system that satisfies the usual
conditions. In particular, all (bi)linear forms are uniformly bounded and
duality $\dual{\dpsi}{u}_\cS$ satisfies the inf-sup condition with constant $1$ by
Lemma~\ref{la_tr}. Furthermore, by Lemma~\ref{la_reg} we have the kernel representation
\[
   \{u\in\sH(A,\mesh);\; \dual{\dpsi}{u}_\cS=0\ \forall \dpsi\in H(A^*,\cS)\}
   = H_0(A)
\]
and the $H_0(A)$-coercivity of $\vdual{\cC\,\cdot}{\cdot}$ holds by assumption \eqref{PF}.
This proves the well-posedness of \eqref{gp}.
Lemma~\ref{la_reg} and relation \eqref{gpb} imply that $u\in H_0(A)$.
Relation~\eqref{gpa} with $\du\in C_0^\infty(\Omega)$
and an application of Lemma~\ref{la_reg} to conclude that $\dual{\psi}{\du}_\cS=0$ for
such $\du$, show that $A^*\cC A u=f$. Using this relation together with the definition
of $\trAsS$, and again \eqref{gpa}, we find that
\[
   \dual{\trAsS(\cC Au)}{\du}_\cS
   = \vdual{f}{\du}_\mesh - \vdual{\cC Au}{A\du}
   = \dual{\phi}{\du}_\cS
   \quad\forall \du\in\sH(A,\mesh),
\]
that is, $\phi=\trAsS(\cC Au)$. This finishes the proof.

\subsection{Proof of Theorem~\ref{thm_gm}} \label{pf_thm_gm}

Again, problem \eqref{gm} is a mixed system that satisfies the usual
conditions. All (bi)linear forms are uniformly bounded.
By Lemma~\ref{la_infsup}, the bilinear form
\[
   b(z;\du,\dphi) := \vdual{A^*z}{\du}_\mesh + \dual{\dphi}{z}_\cS
\]
satisfies the inf-sup condition.
By Lemma~\ref{la_reg} we have the kernel representation
\[
   \{z\in\sH(A^*,\mesh);\; \dual{\dphi}{z}_\cS=0\ \forall \dphi\in H_0(A,\cS)\} = H(A^*)
\]
so that
\begin{align*}
   \ker(b) :=
   &\{z\in\sH(A^*,\mesh);\;
      b(z;\du,\dphi)=0\ \forall \du\in L_2(\Omega),\ \forall\dphi\in H_0(A,\cS)\}\\
   = &\{z\in H(A^*);\; A^*z=0\}.
\end{align*}
The coercivity
\[
   \vdual{\cCinv z}{z} \ge C \|z\|_{A^*,\mesh}^2\quad\forall z\in \ker(b)
\]
with a constant $C>0$ independent of $z$ and $\mesh$ follows.
Therefore, \eqref{gm} is well posed.

Relation $w=\cC Au$ follows from \eqref{gma} by a distributional argument,
also implying that $u\in H(A)$.
Then, \eqref{gma} shows that
\[
   \dual{\trAS(u)}{\dw}_\cS = \vdual{\dw}{Au} - \vdual{A^*\dw}{u}_\mesh
   = \dual{\phi}{\dw}_\cS
   \quad\forall \dw\in\sH(A^*,\mesh),
\]
that is, $\phi=\trAS(u)$. This implies that $u\in H_0(A)$.
Indeed, since $\phi\in H_0(A,\cS)$ we find with Lemma~\ref{la_reg} that
\[
   \dual{\trA(u)}{\dw}_\Gamma = \vdual{Au}{\dw}-\vdual{u}{A^*\dw}_\mesh
   =\dual{\trAS(u)}{\dw}_\cS = \dual{\phi}{\dw}_\cS = 0
   \quad\forall\dw\in H(A^*),
\]
that is, $\trA(u)=0$.

By relation \eqref{gmb} and Lemma~\ref{la_reg}
we have that $w\in H(A^*)$ and $A^*w=f$.
This finishes the proof.

\subsection{Proof of Proposition~\ref{prop_split_h}} \label{pf_prop_split_h}

We use inf-sup stability \eqref{infsup}, the properties of operator $\opF_1$, and relation
$\dual{\dphi}{w}_\cS=0$ for $\dphi\in H_0(A,\cS)$ and $w\in H(A^*)$ by Lemma~\ref{la_reg},
to deduce the bound
\begin{align} \label{infsup_h1}
   \cis \|\du\|
   &\le \sup_{w\in H(A^*)\setminus\{0\}} \frac {\vdual{A^* w}{\du}}{\|w\|_{A^*}}
   \le C_1 \sup_{w\in H(A^*)\cap \sH_h(A^*,\mesh)\setminus\{0\}}
       \frac {\vdual{A^* w}{\du}+\dual{\dphi}{w}_\cS}{\|w\|_{A^*}} \nonumber\\
   &\le C_1 \sup_{w\in \sH_h(A^*,\mesh)\setminus\{0\}}
       \frac {\vdual{A^* w}{\du}_\mesh+\dual{\dphi}{w}_\cS}{\|w\|_{A^*,\mesh}}
   \quad\forall \dphi\in H_h(A,\cS),\ \forall\du\in H_h(\mesh).
\end{align}
Lemma~\ref{la_tr}, the properties of operator $\opF_2$, and estimate \eqref{infsup_h1}
show that any $\dphi\in H_h(A,\cS)$ and $\du\in H_h(\mesh)$ satisfy
\begin{align} \label{infsup_h2}
   \|\dphi\|_{A,\cS}
   &= \sup_{w\in\sH(A^*,\mesh)\setminus\{0\}} \frac {\dual{\dphi}{w}_\cS}{\|w\|_{A^*,\mesh}}
   \le C_2 \sup_{w\in\sH_h(A^*,\mesh)\setminus\{0\}}
           \frac {\dual{\dphi}{w}_\cS}{\|w\|_{A^*,\mesh}} \nonumber\\
   &= C_2 \sup_{w\in\sH_h(A^*,\mesh)\setminus\{0\}} \Bigl(
           \frac {\vdual{A^* w}{\du}_\mesh + \dual{\dphi}{w}_\cS}{\|w\|_{A^*,\mesh}}
         - \frac {\vdual{A^* w}{\du}_\mesh}{\|w\|_{A^*,\mesh}}\Bigr) \nonumber\\
   &\le C_2 \sup_{w\in\sH_h(A^*,\mesh)\setminus\{0\}}
           \frac {\vdual{A^* w}{\du}_\mesh + \dual{\dphi}{w}_\cS}{\|w\|_{A^*,\mesh}}
    + C_2 \|\du\| \nonumber\\
   &\le C_2\bigl(1+\frac {C_1}{\cis}\Bigr) 
        \sup_{w\in\sH_h(A^*,\mesh)\setminus\{0\}}
        \frac {\vdual{A^* w}{\du}_\mesh + \dual{\dphi}{w}_\cS}{\|w\|_{A^*,\mesh}}.
\end{align}
A combination of estimates \eqref{infsup_h1} and \eqref{infsup_h2} proves the
discrete inf-sup property \eqref{infsup_h}.

\subsection{Proof of Theorem~\ref{thm_gu}} \label{pf_thm_gu}

System \eqref{gu} is not of a (standard) mixed form but satisfies the standard properties
of an operator equation \cite{Necas_67_MDT,Babuska_71_EBF}.
As before, all (bi)linear forms are uniformly bounded.
Furthermore, the operator
\[
   \cB:\;\cU:=L_2(\Omega)\times L_2(\Omega;U)\times H_0(A,\cS)\times H(A^*,\cS)
   \to \cV^*:=\sH(A,\mesh)^*\times \sH(A^*,\mesh)^*
\]
defined by the system is injective and satisfies the inf-sup condition,
as we briefly recall now.

{\bf Injectivity.} Let $\dbv=(\du,\dw)\in\cV$ satisfy $\cB^*\dbv=0$.
Lemma~\ref{la_reg} shows that $\dw\in H(A^*)$ and $\du\in H_0(A)$.
Then $\cB^*\dbv=0$ implies $A^*\dw=0$, $\cCinv\dw=A\du$. We conclude that
$\du\in H_0(A)$ solves $A^*\cC A\du=0$, thus $\du=0$ by \eqref{PF}, and $\dw=0$.

{\bf Inf-sup property.}
In abstract form, we have to show that there is a constant $C>0$, independent of $\mesh$ and $\bu$,
that satisfies
\[
   \sup_{\bv\in\cV,\ \|\bv\|_V=1}
   \dual{\cB\bu}{\bv}_{\cV^*\times \cV}
   \ge C \|\bu\|_\cU\quad\forall\bu\in\cU
\]
with (squared) norms
$\|\bu\|_\cU^2:= \|u\|^2 + \|w\|^2 + \|\phi\|_{A,\cS}^2 + \|\psi\|_{A^*,\cS}^2$
and
$\|\bv\|_\cV^2:=\|\du\|_{A,\mesh}^2+\|\dw\|_{A^*,\mesh}^2$
for $\bu\in\cU$ and $\bv=(\du,\dw)\in\cV$.
As Lemma~\ref{la_infsup}, this can be seen by \cite[Theorem~3.3]{CarstensenDG_16_BSF}.
We set
\begin{align*}
   &Y:=\cV,\quad
   Y_0:=H_0(A)\times H(A^*),\quad
   X_0:= L_2(\Omega)\times L_2(\Omega;U),\quad
   \hat X:= H_0(A,\cS)\times H(A^*,\cS),\\
   &b_0((u,w),(\du,\dw)):=\vdual{A\du-\cCinv\dw}{w}_\mesh + \vdual{A^*\dw}{u}_\mesh,\\
   &\hat b((\phi,\psi),(\du,\dw)):= \dual{\phi}{\dw}_\cS + \dual{\psi}{\du}_\cS
   \quad\text{for}\ (u,w)\in X_0,\ (\phi,\psi)\in \hat X,\ (\du,\dw)\in Y
\end{align*}
and need to check conditions \eqref{A1} and \eqref{A2} in the current setting.
Identity \eqref{A2a} holds by Lemma~\ref{la_reg}
and Lemma~\ref{la_tr} implies \eqref{A2b} with $c=1$.
Inf-sup condition \eqref{A1} follows by the stability of the adjoint problem:
\emph{Given $(g,G)\in X_0$ find $(\du,\dw)\in Y$ such that}
\begin{alignat*}{3}
   &A^*\dw &&=g,\quad && A\du-\cCinv\dw = G\\
   \Leftrightarrow\quad
   &A^*\cC A\du &&= g+A^*\cC G,\quad &&\dw =\cC A\du - \cC G.
\end{alignat*}
Of course, this is the initial (self-adjoint) problem \eqref{prob} with general data.
By Assumption \eqref{PF} it is well posed with solution
$(\du,\dw)\in H_0(A)\times H(A^*)$ bounded as
\begin{align*}
   &\|\du\|_A \le C \bigl(\|g\|^2 + \|G\|^2\bigr)^{1/2} = C \|(g,G)\|_{X_0},\\
   &\|\dw\|_{A^*}^2 = \|\dw\|^2 + \|A^*\dw\|^2
   = \|\cC A\du - \cC G\|^2 + \|g\|^2
   \le C^2 \|(g,G)\|_{X_0}^2
\end{align*}
with a constant $C>0$ that depends on $\cC$ but is independent of $g$ and $G$.
The inf-sup property follows by \cite[Theorem~3.3]{CarstensenDG_16_BSF}.

We conclude that problem \eqref{gu} is well posed with solution $(u,w,\phi,\psi)\in\cU$
that satisfies the claimed stability estimate.
Distributional arguments verify the properties $u\in H(A)$, $w=\cC A u\in H(A^*)$,
and $A^*w=f$.
Then, applying \eqref{gua} and \eqref{gub}, we find, respectively, that
\[
   \dual{\trAS(u)}{\dw}_\cS =
   \vdual{Au}{\dw} - \vdual{u}{A^*\dw}_\mesh =
   \vdual{\cCinv w}{\dw} - \vdual{u}{A^*\dw}_\mesh = \dual{\phi}{\dw}_\cS
\]
for any $\dw\in \sH(A^*,\mesh)$ and
\[
   \dual{\trAsS(w)}{\du}_\cS =
   \vdual{A^*w}{\du} - \vdual{w}{A\du}_\mesh =
   \vdual{A^*w}{\du} - \vdual{f}{\du} + \dual{\psi}{\du}_\cS = \dual{\psi}{\du}_\cS
\]
for any $\du\in\sH(A,\mesh)$.
It follows that $\trAS(u)=\phi$, also implying $u\in H_0(A)$, and $\trAsS(w)=\psi$.
This finishes the proof.


\bibliographystyle{siam}
\bibliography{/home/norbert/tex/bib/bib,/home/norbert/tex/bib/heuer}

\end{document}